\documentclass[12pt,reqno,twoside]{article}
\usepackage{amssymb,amsfonts,amsthm,amsmath}
\usepackage{cite}
\usepackage[left=1 in,top=1 in,right=1 in,bottom=1 in]{geometry}
\usepackage{enumitem}

\usepackage{color}

\usepackage{hyperref}

\newcommand{\beq}{\begin{equation}}
\newcommand{\eeq}{\end{equation}}
\newcommand{\beqs}{\begin{equation*}}
\newcommand{\eeqs}{\end{equation*}}
\newcommand{\ba}{\begin{array}}
\newcommand{\ea}{\end{array}}
\newcommand{\beas}{\begin{eqnarray*}}
\newcommand{\eeas}{\end{eqnarray*}}
\newcommand{\bea}{\begin{eqnarray}}
\newcommand{\eea}{\end{eqnarray}}
\newcommand{\bal}{\begin{align}}
\newcommand{\eal}{\end{align}}

\newcommand{\bals}{\begin{align*}}
\newcommand{\eals}{\end{align*}}

\newcommand{\al}{\alpha}

\newcommand{\R}{\ensuremath{\mathbb R}}

\newcommand{\C}{\ensuremath{\mathbb C}}
\newcommand{\N}{\ensuremath{\mathbb N}}
\newcommand{\Z}{\ensuremath{\mathbb Z}}

\newcommand{\bigo}{\mathcal O}

\newcommand{\inprod}[1]{\langle{#1}\rangle}
\newcommand{\doubleinprod}[1]{\langle\!\langle{#1}\rangle\!\rangle}

\newcommand{\bds}{\begin{displaystyle}}
\newcommand{\eds}{\end{displaystyle}}

\newcommand{\goodt}{\mathcal T}

\def\eqdef{\stackrel{\rm def}{=}}
\def\Oeq{\stackrel{\mathcal O}{=}}

\def\simP{\stackrel{\Psi}{\sim}}
\def\dsim#1{\underset{#1}{\sim}}

\newcommand{\bvec}[1]{\mathbf{#1}}
\def\vecx{\bvec x}

\def\vece{\bvec e}
\def\vecu{\bvec u}
\def\vecv{\bvec v}
\def\vecx{\bvec x}
\def\vecw{\bvec w}

\def\veck{\bvec k}

\def\varep{\varepsilon}
\def\ddt{\frac{d}{dt}}

\newtheorem{theorem}{Theorem}[section]

\newtheorem{lemma}[theorem]{Lemma}
\newtheorem{corollary}[theorem]{Corollary}
\newtheorem{proposition}[theorem]{Proposition}
\newtheorem{definition}[theorem]{Definition}
\newtheorem{assumption}[theorem]{Assumption}
\newtheorem{condition}[theorem]{Condition}

\theoremstyle{remark}
\newtheorem{remark}[theorem]{\bf{Remark}}
\newtheorem{example}[theorem]{\bf{Example}}
\newtheorem*{notation}{\bf{Notation}}

\def\mD{\mathcal D}

\numberwithin{equation}{section}

\title{\textbf{Asymptotic expansions in a general system of decaying functions for solutions of the Navier-Stokes equations}}
\author{Dat Cao and Luan Hoang}

\date{\today}
\begin{document}
\maketitle

\begin{center}
Department of Mathematics and Statistics, Texas Tech University\\
Box 41042, Lubbock, TX 79409-1042, U.S.A.\\
Email addresses: \texttt{dat.cao@ttu.edu, luan.hoang@ttu.edu}
\end{center}

\begin{abstract}
We study the long-time dynamics of the Navier-Stokes equations in the three-dimensional periodic domains with a   body force decaying in time. We introduce appropriate systems of decaying functions and corresponding asymptotic expansions in those systems. We prove that if the force has a large-time asymptotic expansion in Gevrey-Sobolev spaces in such a general system, then any Leray-Hopf weak solution admits an asymptotic expansion of the same type. This expansion is uniquely determined by the force, and independent of the solutions. 
Various applications of the abstract results are provided which particularly include the previously obtained expansions for the solutions in case of power decay, as well as the new expansions in case of the  logarithmic and iterated logarithmic decay.\end{abstract}
\tableofcontents

\section{Introduction}\label{intro}
We study the long-time behavior of viscous, incompressible fluid flows in  space $\R^3$.
First, we recall the Navier-Stokes equations (NSE) that describe the fluid dynamics.

Let $\vecx\in \R^3$ and $t\in\R$  denote the space and time variables, respectively.
Let the (kinematic) viscosity be denoted by $\nu>0$, the  velocity vector field by $\vecu(\vecx,t)\in\R^3$, the pressure by $p(\vecx,t)\in\R$, and the body force by  $\mathbf f(\vecx,t)\in\R^3$. The NSE  are 
\begin{align}\label{nse}
\begin{split}
&\bds \frac{\partial \vecu}{\partial t}\eds  + (\vecu\cdot\nabla)\vecu -\nu\Delta \vecu = -\nabla p+\mathbf f \quad\text{on }\R^3\times(0,\infty),\\
&\textrm{div } \vecu = 0  \quad\text{on }\R^3\times(0,\infty).
\end{split}
\end{align}

The initial condition is 
\beq\label{ini}
\vecu(\vecx,0) = \vecu^0(\vecx),
\eeq 
where  $\vecu^0(\vecx)$ is a given divergence-free vector field. 
 
We avoid the unbounded domains and the boundary conditions by considering only force $\mathbf f(\vecx,t)$ and solutions $(\vecu(\vecx, t),p(\vecx, t))$ that are $L$-periodic for some $L>0$.
Hereafter, a function $g(\vecx)$ is said to be $L$-periodic if
\beqs
g(\vecx+L\vece_j)=g(\vecx)\quad \textrm{for all}\quad \vecx\in \R^3,\ j=1,2,3,\eeqs
where  $\{\vece_1,\vece_2,\vece_3\}$ is the standard basis of $\R^3$, and is said to have zero average over the domain $\Omega=(-L/2, L/2)^3$ if 
\beqs
\int_\Omega g(\vecx)d\vecx=0.
\eeqs

By using a particular Galilean transformation, see details in, e.g., \cite{HM2}, we can also assume, for all $t\ge 0$, that $\mathbf f(\vecx,t)$ and $\vecu(\vecx,t)$,  have zero averages over the domain $\Omega$.
In light of the Leray-Helmholtz decomposition, and for the sake of convenience, we also assume that $\mathbf f(\vecx,t)$ is divergence-free for all $t\ge 0$. 

By rescaling the variables $\vecx$ and $t$, we assume throughout, without loss of generality, that  $L=2\pi$ and $\nu =1$.

Throughout the paper, we use the following notation  
$$u(t)=\vecu(\cdot,t),\  f(t)=\mathbf f(\cdot,t),\  u^0=\vecu^0(\cdot),$$ 
which are function-valued.

\medskip
In the case of potential force, that is, $\mathbf f(\vecx,t)=-\nabla \phi(\vecx,t)$, for some scalar function $\phi$, it is 
Foias and Saut proved in  \cite{FS87}  that  any  non-trivial, regular solution
$u(t)$ in bounded or periodic domains admits an  asymptotic expansion (as $t\to\infty$)
\beq \label{expand}
u(t) \sim \sum_{n=1}^\infty q_n(t)e^{-nt}
\eeq
in Sobolev spaces $H^m(\Omega)^3$, for all $m\ge 0$.
The interested reader is referred to \cite{DE1968,FS84a} for early results on the solutions' asymptotic behavior,
\cite{FS83,FS84a,FS84b,FS87,FS91} for associated normalization map and invariant nonlinear manifolds, 
\cite{FHOZ1,FHOZ2,FHS1} for the corresponding Poincar\'e-Dulac normal form,
\cite{FHN1,FHN2} for their applications to analysis of helicity, statistical solutions, and decaying turbulence.
The recent paper \cite{FHS2} is a survey on the subject. 

In case of periodic domains, it was then improved in \cite{HM1} that the expansion holds in any Gevrey spaces $G_{\alpha,\sigma}$, see Section \ref{Prelim} for details. When the force $f$ is not potential, the asymptotic expansion of Leray-Hopf weak solutions is established in \cite{HM2} for an exponentially decaying force: if the force has an asymptotic expansion 
\beq\label{fexp0}
f(t)\sim \sum_{n=1}^\infty p_n(t)e^{-nt},
\eeq
then $u(t)$ has an asymptotic expansion of type \eqref{expand}.

The case of power-decaying forces is treated in \cite{CaH1}: if 
\beq\label{fpow}
f(t)\sim \sum_{n=1}^\infty \phi_n t^{-\gamma_n},
\eeq
then all Leray-Hopf weak solutions $u(t)$ amid the same expansion 
\beq\label{upow}
u(t)\sim \sum_{n=1}^\infty \xi_n t^{-\mu_n}.
\eeq
Above $\phi_n$'s and $\xi_n$'s belong to some Gevrey-Sobolev space. The meanings of the expansions \eqref{expand}, \eqref{fexp0}, \eqref{fpow}, \eqref{upow} are specified precisely in the cited papers.

The current paper aims to develop the results in \cite{CaH1} to cover a very large class of forces. For example, we will prove that if
\beqs
f(t)\sim \sum_{n=1}^\infty \phi_n (\ln t)^{-\gamma_n},\text{ or , }
f(t)\sim \sum_{n=1}^\infty \phi_n (\ln(\ln t))^{-\gamma_n},
\eeqs
then 
\beqs
u(t)\sim \sum_{n=1}^\infty \xi_n (\ln t)^{-\mu_n}, \text{ or respectively, }
u(t)\sim \sum_{n=1}^\infty \xi_n (\ln(\ln t))^{-\mu_n}. 
\eeqs

In fact, we obtain a much more general result which is described very roughly here.
Let $(\psi_n)_{n=1}^\infty$ be a sequence of time decaying functions with $\psi_{n+1}(t)$ decays to zero, as $t\to\infty$, much faster than $\psi_n(t)$. The functions $\psi_n$'s are assumed to satisfy a certain set of conditions.

Suppose  there exist $\alpha\ge 1/2$ and $\sigma\ge 0$ such that 
\beq\label{fpol}
f(t)\sim \sum_{n=1}^\infty \phi_n\psi_n(t)\quad\text{in } G_{\alpha,\sigma}.
\eeq

We will prove that any Leray-Hopf weak solution $u(t)$ will admit an expansion
\beq\label{upol}
u(t)\sim \sum_{n=1}^\infty \xi_n \psi_n(t) \quad\text{in } G_{\alpha,\sigma},
\eeq
where $\xi_n$'s are explicitly determined by $\phi_n$'s. The meaning of the expansions \eqref{fpol} and \eqref{upol} is  more sophisticated than \eqref{expand}--\eqref{upow}, thanks to their generality, and will be made clear later in the paper.

The paper is organized as follows.
Section \ref{Prelim} reviews the functional setting for the NSE which is suitable for studying the solutions' dynamics in time.
It also recalls basic inequalities for the Stokes operator and the bilinear form in the NSE.
In Section \ref{estsec}, we  establish the long-time estimates for solutions of both linearized NSE (subsection \ref{linearNSE}) and the NSE (subsection \ref{LHdecay}) with very general decaying forces. The results for the linearized NSE play a key role in  improving the long-time estimates for the solutions of the NSE. 
Having their own merits, these estimates are also crucial to the proofs in Sections \ref{expsec} and \ref{discrtsec}.
Section \ref{syssec} introduces the definitions of systems of decaying functions in time and the asymptotic expansions in those systems.
In Definition \ref{premsys}, we aim to balance between the generality, such as in Definition \ref{sys0}, and the technical requirements.
Condition \ref{sysdef} is particularly emphasized on applications to ordinary and partial differential equations with quadratic or integral power nonlinearity. Condition \ref{weaksys} is focused on functions which are larger than the exponential ones.
We state and prove elementary properties for these systems and expansions. 
In Section \ref{expsec}, we  obtain in Theorem \ref{mainthm} the expansions in Gevrey-Sobolev spaces for all Leray-Hopf solutions of the NSE, when a continuum system of decaying functions is available as the expansions' basis. The result gives precise meanings to the above expansions \eqref{fpol} and \eqref{upol}.
A version of finite sum asymptotic approximations is proved in Theorem \ref{finthm1}.
It is suitable for a force that has limited information about their long-time behavior. 
In Section \ref{discrtsec}, we study the situation when the discrete system of functions for expansions cannot be embedded directly into a continuum system as in Section \ref{expsec}. However, by using a continuum background system, we can still obtain in Theorem \ref{discretethm} the asymptotic expansions for solutions of the NSE. An asymptotic approximation result for the discrete system is similarly obtained in Theorem \ref{finthm2}.
%
Section \ref{examsec} provides many applications of the abstract results in sections \ref{expsec} and \ref{discrtsec}. 
They  consist of the recovery of the power decay case previously established in \cite{CaH1}, see subsection \ref{ss1},  as well as the new logarithmic and iterated logarithmic decay cases, see Theorem \ref{logcase} and Corollary \ref{purelog}. 
Examples \ref{sin-eg} and \ref{tan-eg} demonstrate some asymptotic  expansions with trigonometric functions.
More complicated expansions are presented in Propositions  \ref{rootcase} and \ref{productcase}, particularly, the latter one can only be achieved by using of the background systems developed in Section \ref{discrtsec}. 
%
Appendix \ref{apA} contains some criteria for a convergent series of functions to have corresponding asymptotic expansions of the types  specified in Sections \ref{syssec} and \ref{discrtsec}.

\section{Functional setting and basic facts for the NSE}\label{Prelim}

We recall the standard functional setting for the NSE, see e.g. \cite{CFbook, TemamAMSbook,TemamSIAMbook, FMRTbook}, and some basic inequalities and estimates.

Let $L^2(\Omega)$ and $H^m(\Omega)=W^{m,2}(\Omega)$, for integers $m\ge 0$, denote the standard Lebesgue and Sobolev spaces on $\Omega$.
The standard inner product and norm in $L^2(\Omega)^3$ are denoted by $\inprod{\cdot,\cdot}$ and $|\cdot|$, respectively.
(We warn that this  notation  $|\cdot|$ also denotes the Euclidean norm in $\R^n$ and $\C^n$, for any $n\in\N$, but its meaning will be clear based on the context.)

Let $\mathcal{V}$ be the set of all $2\pi$-periodic trigonometric polynomial vector fields which are divergence-free and  have zero average over $\Omega$.  
Define
$$H, \text{ resp. } V\ =\text{ closure of }\mathcal{V} \text{ in }
L^2(\Omega)^3, \text{ resp. } H^1(\Omega)^3.$$
Notice that each element of $H$ is divergence-free and has zero average over $\Omega$,
 and each element of $V$ is $2\pi$-periodic.

We use the following embeddings and identification
$$V\subset H=H'\subset V',$$ 
where each space is dense in the next one, and the embeddings are compact.

Let $\mathcal{P}$ denote the orthogonal (Leray) projection in $L^2(\Omega)^3$ onto $H$.

The Stokes operator $A$ is a bounded linear mapping from $V$ to its dual space $V'$ defined by  
\beqs
\inprod{A\vecu,\vecv}_{V',V}=
\doubleinprod{\vecu,\vecv}
\eqdef \sum_{j=1}^3 \inprod{ \frac{\partial \vecu}{\partial x_j} , \frac{\partial \vecv}{\partial x_j} }\quad \text{for all } \vecu,\vecv\in V.
\eeqs

As an unbounded operator on $H$, the operator $A$ has the domain $\mD(A)=V\cap H^2(\Omega)^3$, and, under  the current consideration of periodicity conditions, 
\beqs A\vecu = - \mathcal{P}\Delta \vecu=-\Delta \vecu\in H \quad \textrm{for all}\quad \vecu\in\mD(A).
\eeqs

The spectrum  of $A$ is known to be 
$$\sigma(A)=\{|\veck|^2: \ \veck\in\Z^3, \veck\ne \mathbf 0\},$$
and each $\lambda\in \sigma(A)$ is an eigenvalue.
Note that $\sigma(A)\subset \N$ and $1\in\sigma(A)$, hence, the additive semigroup generated by $\sigma(A)$ is $\N$.

For $n\in\sigma (A)$, we denote by  $R_n$ the orthogonal projection in $H$ on the eigenspace of $A$ corresponding to $n$,
and set $$P_n=\sum_{j\in \sigma(A),j\le n}R_j.$$ 
Note that each vector space $P_nH$ is finite dimensional.

\medskip 
For $\alpha,\sigma \in \R$ and  $\vecu=\sum_{\veck\ne \mathbf 0} 
\widehat \vecu(\veck)e^{i\veck\cdot \vecx}$, define
$$A^\alpha \vecu=\sum_{\veck\ne \mathbf 0} |\veck|^{2\alpha} \widehat \vecu(\veck)e^{i\veck\cdot 
\vecx},\quad e^{\sigma A^{1/2}} \vecu=\sum_{\veck\ne \mathbf 0} e^{\sigma 
|\veck|} \widehat \vecu(\veck)e^{i\veck\cdot 
\vecx},$$
and, hence,
$$A^\alpha e^{\sigma A^{1/2}} \vecu=e^{\sigma A^{1/2}}A^\alpha  \vecu=\sum_{\veck\ne \mathbf 0} |\veck|^{2\alpha}e^{\sigma 
|\veck|} \widehat \vecu(\veck)e^{i\veck\cdot 
\vecx}.$$

The  Gevrey spaces are defined by
\beqs
G_{\alpha,\sigma}=\mD(A^\alpha e^{\sigma A^{1/2}} )\eqdef \{ \vecu\in H: |\vecu|_{\alpha,\sigma}\eqdef |A^\alpha 
e^{\sigma A^{1/2}}\vecu|<\infty\},
\eeqs
and, in particular, when $\sigma=0$,  the domain of the fractional operator $A^\al$ is 
\beqs 
\mD(A^\alpha)=G_{\alpha,0}=\{ \vecu\in H: |A^\alpha \vecu|=|\vecu|_{\alpha,0}<\infty\}.
\eeqs

Thanks to the zero-average condition, the norm $|A^{m/2}\vecu|$ is equivalent to $\|\vecu\|_{H^m(\Omega)^3}$ on the space $\mathcal D(A^{m/2})$ for $m=0,1,2,\ldots$

Note that $\mD(A^0)=H$,  $\mD(A^{1/2})=V$, and  $\|\vecu\|\eqdef |\nabla \vecu|$ is equal to $|A^{1/2}\vecu|$ for $\vecu\in V$.
Also, the norms $|\cdot|_{\alpha,\sigma}$ are increasing in $\alpha$, $\sigma$, hence, the spaces $G_{\alpha,\sigma}$ are decreasing in $\alpha$, $\sigma$.

\medskip
Regarding the nonlinear term in the NSE, a bounded linear map $B:V\times V\to V'$ is defined by
\beqs
\inprod{B(\vecu,\vecv),\vecw}_{V',V}=b(\vecu,\vecv,\vecw)\eqdef \int_\Omega ((\vecu\cdot \nabla) \vecv)\cdot \vecw\, d\vecx, \quad \textrm{for all}\quad \vecu,\vecv,\vecw\in V.
\eeqs 
In particular,
\beqs
B(\vecu,\vecv)=\mathcal{P}((\vecu\cdot \nabla) \vecv), \quad \textrm{for all}\quad \vecu,\vecv\in\mD(A).
\eeqs

The problems  \eqref{nse} and \eqref{ini} can now be rewritten in the functional form as
\begin{align}\label{fctnse}
&\frac{du(t)}{dt} + Au(t) +B(u(t),u(t))=f(t) \quad \text{ in } V' \text{ on } (0,\infty),\\
\label{uzero} 
&u(0)=u^0\in H.
\end{align}
(We refer the reader to the books \cite{LadyFlowbook69,CFbook,TemamAMSbook,TemamSIAMbook} for more details.)


The next definition makes precise the meaning of weak solutions of \eqref{fctnse}.

\begin{definition}\label{lhdef}
Let $f\in L^2_{\rm loc}([0,\infty),H)$.
A \emph{Leray-Hopf weak solution} $u(t)$ of \eqref{fctnse} is a mapping from $[0,\infty)$ to $H$ such that 
\beqs
u\in C([0,\infty),H_{\rm w})\cap L^2_{\rm loc}([0,\infty),V),\quad u'\in L^{4/3}_{\rm loc}([0,\infty),V'),
\eeqs
and satisfies 
\beq\label{varform}
\ddt \inprod{u(t),v}+\doubleinprod{u(t),v}+b(u(t),u(t),v)=\inprod{f(t),v}
\eeq
in the distribution sense in $(0,\infty)$, for all $v\in V$, and the energy inequality
\beq\label{Lenergy}
\frac12|u(t)|^2+\int_{t_0}^t \|u(\tau)\|^2d\tau\le \frac12|u(t_0)|^2+\int_{t_0}^t \langle f(\tau),u(\tau)\rangle d\tau
\eeq
holds for $t_0=0$ and almost all $t_0\in(0,\infty)$, and all $t\ge t_0$.  
Here, $H_{\rm w}$ denotes the topological vector space $H$ with the weak topology.
 
A \emph{regular} solution is a Leray-Hopf weak solution that belongs to $C([0,\infty),V)$.

If $t\mapsto u(T+t)$ is a Leray-Hopf weak/regular solution, then we say $u$ is a Leray-Hopf weak/regular solution on $[T,\infty)$.
\end{definition}

It is well-known that a regular solution is unique among all Leray-Hopf weak solutions.

We denote by $\goodt$ the set of $t_0\in[0,\infty)$ such that \eqref{Lenergy} holds for all $t\ge t_0$. Note that $[0,\infty)\setminus \goodt$ has zero measure.

If $u(t)$ is a Leray-Hopf weak solution and $t_0\in\goodt$, then $u(t_0+t)$ is also a Leray-Hopf weak solution.
Assume additionally that there exists a regular solution $w(t)$ with $w(0)=u(t_0)$. Then by the uniqueness of $w(t)$, we have $u(t_0+t)=w(t)$ and, hence, $u(t)$ is a regular solution on $[t_0,\infty)$. 

\medskip
We assume throughout the paper that 

\begin{assumption}\label{A1} The function $f$ belongs to $L^\infty_{\rm loc}([0,\infty),H)$.
\end{assumption}

Under Assumption \ref{A1}, for any $u^0\in H$, there exists a Leray-Hopf weak solution $u(t)$ of \eqref{fctnse} and \eqref{uzero}, see e.g. \cite{FMRTbook}. 
The large-time behavior of $u(t)$ is the focus of our study.
More specific conditions on $f$ will be imposed later. 

We note that, thanks to Remark 1(e) of  \cite{FRT2010}, the Leray-Hopf weak solutions in Definition \ref{lhdef} are the same as the weak solutions defined in \cite[Chapter II, section 7]{FMRTbook}. Hence, according to inequality (A.39) in \cite[Chap. II]{FMRTbook}, we have for any such solution $u(t)$ that
\beq\label{iniener}
|u(t)|^2\le e^{-t}|u(0)|^2 + \int_0^t e^{-(t-\tau)} |f(\tau)|^2d\tau\quad \forall t>0.
\eeq

Below are some inequalities that will be needed in later estimates.
First, for any $\sigma,\alpha>0$, one has
\beq\label{mxe}
\max_{x\ge 0} (x^{\alpha}e^{-\sigma x})=d_0(\alpha,\sigma)\eqdef \Big(\frac{\alpha}{e\sigma }\Big)^{\alpha},
\eeq 
and, hence, 
\beqs\label{mx2}
e^{-\sigma x}=e^{-\sigma (x+1)}e^\sigma\le d_0(\alpha,\sigma)e^\sigma (1+x)^{-\alpha}\quad \forall x\ge 0.
\eeqs

Thanks to \eqref{mxe}, one can verify, for all $\alpha,\sigma>0$, that
\beq\label{als0}
|A^\alpha e^{-\sigma A}v| \le d_0(\alpha,\sigma) |v|\quad \forall v\in H,
\eeq
\beqs
|A^\alpha e^{-\sigma A^{1/2}}v| \le d_0(2\alpha,\sigma) |v|\quad \forall v\in H,
\eeqs 
and, consequently, 
\beq 
\label{als}
|A^\alpha v|=|(A^\alpha e^{-\sigma A^{1/2}})e^{\sigma A^{1/2}}v|\le d_0(2\alpha,\sigma) |e^{\sigma A^{1/2}}v|\quad  \forall v\in G_{0,\sigma}.
\eeq

For the bilinear mapping $B(u,v)$, it follows from its boundedness that there exists a constant $K_*>0$ such that 
\beq\label{Bweak}
\|B(u,v)\|_{V'}\le  K_* \| u\|\, \|v\|\quad\forall u,v\in V.
\eeq

The estimate of $B(u,v)$'s Gevrey norms was initiated by Foias-Temam \cite{FT-Gevrey}.
Here we recall an extended and convenient (though not sharp) version from \cite[Lemma 2.1]{HM1}.  

\medskip
\textit{There exists a constant $K>1$ such that if $\sigma\ge 0$ and $\alpha\ge 1/2$, then 
\beq\label{AalphaB} 
|B(u,v)|_{\alpha ,\sigma }\le K^\alpha  |u|_{\alpha +1/2,\sigma } \, |v|_{\alpha +1/2,\sigma}\quad \forall u,v\in G_{\alpha+1/2,\sigma}.
\eeq
}

\begin{notation}   We make clear the meaning of the ``big oh'' and ``small oh'' notation.
Let $f$ and $g$ be two non-negative functions defined on a neighborhood of infinity (in $\R$).
\begin{itemize}
 \item We write $f(t)=\mathcal O(g(t))$ (implicitly means as $t\to\infty$) if there exist $T,C>0$, such that $f(t)\le Cg(t)$ for all $t\ge T$.

 \item We say $f(t)\Oeq g(t)$, if $f(t)=\mathcal O(g(t))$ and $g(t)=\mathcal O(f(t))$.

 \item We write $f(t)=o(g(t))$ (implicitly means as $t\to\infty$)  if for any $\varep>0$, there exist $T_\varep>0$, such that $f(t)\le \varep g(t)$ for all $t\ge T_\varep$.
\end{itemize}
\end{notation}

Let $u(t)=\mathcal O(f(t))$ and $v(t)=\mathcal O(g(t))$. Then clearly
$(uv)(t)= \mathcal O(g(t)g(t))$, which we simply write as
\beq\label{O1}
\mathcal O(f(t)) \mathcal O(g(t))=\mathcal O(f(t)g(t)), \text{ and particularly, }
f(t) \mathcal O(g(t))=\mathcal O(f(t)g(t)).
\eeq

If $f(t)=\mathcal O(g(t))$, then $(u+v)(t)=\mathcal O(g(t))$, which  we write as
\beq\label{O2}
\mathcal O(f(t))+\mathcal O(g(t))=\mathcal O(g(t)).
\eeq

Note that the identities in \eqref{O1} and \eqref{O2} are only for convenience, and should be read from left to right. 

\section{Large-time estimates}\label{estsec}

This section contains long-time estimates for solutions of the linearized NSE and of the NSE with the force decaying in time.

First, we have a convenient integral estimate which will be used repeatedly later.

\begin{lemma}\label{Flem}
 Let $F$ be a continuous, decreasing function from $[0,\infty)$ to $[0,\infty)$.
For any $\sigma>0$ and $\theta\in(0,1)$, one has  
\beq \label{Fi1}
 \int_0^t e^{-\sigma (t-\tau)}F(\tau)d\tau
 \le \frac1\sigma\Big(F(0)e^{-(1-\theta)\sigma t}+F(\theta t)\Big)\quad\forall t\ge 0.
\eeq
\end{lemma}
\begin{proof}
We follow the proof of \cite[Lemma 2.2]{CaH1}. We split
\begin{align*}
\int_0^ t e^{-\sigma(t- \tau)} F(\tau) d\tau= I_1+I_2,
\end{align*}
where $I_1$ is the integral from $0$ to $\theta t$, and $I_2$ is the integral from $\theta t$ to $t$.
Using the monotonicity of $F$, we estimate 
\begin{align*}
I_1&\le F(0) \int_0^{\theta t} e^{-\sigma(t-\tau)} d\tau 
 \le  F(0) \frac{e^{-(1-\theta)\sigma t}}\sigma,\\
I_2&\le F(\theta t)\int_{\theta t}^ t e^{-\sigma(t- \tau)}  d\tau
 \le  F(\theta t) \frac{1}\sigma.
\end{align*}
Thus, we obtain \eqref{Fi1}.
 \end{proof}

\subsection{The linearized NSE}\label{linearNSE}

We establish an explicit estimate for solutions of the linearized NSE in terms of the decaying force.

\begin{theorem}\label{Fode}
Given $\alpha,\sigma\ge 0$, let $\xi\in G_{\alpha,\sigma}$, and $f$ be a function from $(0,\infty)$ to $G_{\alpha,\sigma}$ that satisfies 
  \beq\label{fF} |f(t)|_{\alpha,\sigma}\le MF(t) \quad\text{a.e. in $(0,\infty)$.}\eeq 
where $M$ is a positive constant, and $F$ is a continuous, decreasing function from $[0,\infty)$ to $[0,\infty)$.

Let $w_0\in G_{\alpha,\sigma}$. Suppose  $w\in C([0,\infty),H_{\rm w})\cap L^1_{\rm loc}([0,\infty),V) $, with $w'\in L^1_{\rm loc}([0,\infty),V')$,  is a weak solution of 
 \beqs
 w'=-Aw+\xi+f \text{ in $V'$ on $(0,\infty)$,} \quad w(0)=w_0,
  \eeqs
  i.e., it holds, for all $v\in V$, that
  \beqs
 \ddt \inprod{w,v}=-\doubleinprod{w,v}+\inprod{\xi+f,v} \text{ in the distribution sense on $(0,\infty)$.}
  \eeqs
  
  Then the following statements hold true.
\begin{enumerate}[label=\rm (\roman*)]
 \item $w(t)\in G_{\alpha+1-\varep,\sigma}$ for all $\varep\in(0,1)$ and $t>0$.
 
 \item For any numbers $a,a_0\in (0,1)$ with $a+a_0<1$ and any $\varep\in(0,1)$, there exists a positive constant $C$ depending on $a_0$, $a$, $\varep$, $M$, $F(0)$, $|\xi|_{\alpha,\sigma}$ and $|w_0|_{\alpha,\sigma}$ such that 
\beq\label{wremain3}
|w(t)-A^{-1}\xi|_{\alpha+1-\varep,\sigma}\le C\big(e^{-a_0 t}+F(a t)\big)  \quad \forall t \ge 1. 
\eeq
 
 \item Assume, in addition, that 
 \begin{itemize}
  \item There exist $k_0>0$ and $D_1>0$ such that
 \beq\label{eF}
  e^{-k_0 t}\le D_1 F(t)\quad\forall t\ge 0, \text{ and  }
 \eeq
\item For any $a\in(0,1)$, there exists $D_2=D_{2,a}>0$ such that 
\beq\label{Fa}
F(at)\le D_2 F(t)\quad \forall t\ge 0.
\eeq
 \end{itemize}
Then there exists $C>0$ such that
\beq\label{wremain2}
|w(t)-A^{-1}\xi|_{\alpha+1-\varep,\sigma}\le CF(t)  \quad \forall t \ge 1. 
\eeq
\end{enumerate}
\end{theorem}
 \begin{proof}
(i) This regularity result is the same as \cite[Lemma 2.4(i)]{CaH1} in which we set $f:=\xi+f$.

(ii) First, we state and prove a more technical version of \eqref{wremain3}.

\textit{For any $\varep,\delta,\theta,\theta'\in(0,1)$, there exists $C>0$ depending on $\varep$, $\delta$, $\theta$, $M$, $F(0)$, $|\xi|_{\alpha,\sigma}$ and $|w_0|_{\alpha,\sigma}$ such that 
\beq\label{wremain}
|w(t)-A^{-1}\xi|_{\alpha+1-\varep,\sigma}\le C\big(e^{-(1-\theta')\theta \delta t}+F(\theta'\theta t)\big)  \quad \forall t \ge 1. 
\eeq
}

\textit{Proof of \eqref{wremain}.} We follow the proof of \cite[Lemma 2.3]{CaH1}. 

(a) Let $N\in\sigma(A)$. 
We recall formula (2.19) of \cite{CaH1}:
\beq\label{yform}
  P_N\big (w(t)-A^{-1}\xi\big )=e^{-tA}P_N w_0-A^{-1}e^{-tA}P_N \xi+\int_0^t e^{-(t-\tau)A}P_N f(\tau)d\tau\quad\forall t\ge 0.
\eeq
(This formula was derived by using the equation of $P_N w$ and the variation of constant formula.)
  
Let $\varep\in (0,1)$. 
Applying $A^{1-\varep}$ to both sides of \eqref{yform}, and estimating the $|\cdot|_{\alpha,\sigma}$ norm of the resulting quantities,   we obtain
\beq \label{ynorm}
  \begin{aligned}
  |P_N(w(t)-A^{-1}\xi)|_{\alpha+1-\varep,\sigma}
  &\le |A^{1-\varep}e^{-tA}w_0|_{\alpha,\sigma}+|A^{-\varep}e^{-tA}\xi|_{\alpha,\sigma}\\
  &\quad +\int_0^t |e^{-(t-\tau)A}A^{1-\varep} f(\tau)|_{\alpha,\sigma}d\tau.
    \end{aligned}
\eeq  

Let $\theta,\delta\in(0,1)$ and $t\ge 1$. We estimate each term on the right-hand side of \eqref{ynorm} separately.  
  
(b) Rewriting the first term on the right-hand side of \eqref{ynorm} and applying \eqref{als0} yield
  \begin{align*}
   |A^{1-\varep}e^{-tA}w_0|_{\alpha,\sigma}
   &= |A^{1-\varep}e^{-(1-\delta)tA} (e^{-\delta t A} w_0)|_{\alpha,\sigma}
   \le \Big[ \frac{1-\varep}{e(1-\delta)t}\Big]^{1-\varep}|e^{-\delta tA} w_0|_{\alpha,\sigma}\\
   &\le \Big[ \frac{1-\varep}{e(1-\delta)}\Big]^{1-\varep} e^{-\delta t} |w_0|_{\alpha,\sigma}.
\end{align*}

The second term on the right-hand side of \eqref{ynorm} can be easily estimated by
  \beqs
   |A^{-\varep}e^{-tA}\xi|_{\alpha,\sigma}
\le  | e^{-tA} \xi|_{\alpha,\sigma}
   \le e^{-t} |\xi|_{\alpha,\sigma}. 
\eeqs

(c) Dealing with  the last integral in \eqref{ynorm}, we split it into two integrals
\beqs
\int_0^t |e^{-(t-\tau)A}A^{1-\varep}f(\tau)|_{\alpha,\sigma}d\tau=I_1+I_2,
\eeqs
where $I_1$ is the integral from $0$ to $\theta t$, and $I_2$ from $\theta t$ to $t$.

For $I_1$, we have 
\begin{align*}
 I_1&= \int_0^{\theta t} \Big|e^{-(t-\tau)(1-\delta)A}\Big( e^{-(t-\tau)\delta A} A^{1-\varep} f(\tau)\Big)\Big|_{\alpha,\sigma}d\tau\\
 &\le \int_0^{\theta t} \Big|e^{-(1-\theta)t(1-\delta) A}A^{1-\varep} \Big(e^{-(t-\tau)\delta A}f(\tau)\Big)\Big|_{\alpha,\sigma}d\tau.
\end{align*}

Utilizing \eqref{als0} and then using hypothesis \eqref{fF}, we obtain
\begin{align*}
 I_1&\le \int_0^{\theta t} \Big [\frac{1-\varep}{e(1-\theta)(1-\delta)t}\Big ]^{1-\varep}|e^{-(t-\tau)\delta A}f(\tau)|_{\alpha,\sigma}d\tau\\
  &\le \Big [\frac{1-\varep}{e(1-\theta)(1-\delta)t}\Big ]^{1-\varep} \int_0^{\theta t}  e^{-(t-\tau)\delta} MF(\tau)d\tau.\\
 &\le M\Big [\frac{1-\varep}{e(1-\theta)(1-\delta)}\Big ]^{1-\varep} \int_0^{\theta t}  e^{-\delta (\theta t-\tau)} F(\tau)d\tau.
\end{align*} 

Let $\theta'\in(0,1)$. Then by Lemma \ref{Flem},
\begin{align*}
 I_1
 &\le  M \Big [\frac{1-\varep}{e(1-\theta)(1-\delta)}\Big ]^{1-\varep} \frac1\delta \big(F(0) e^{-(1-\theta')\delta \theta t}+F(\theta'\theta t)\big).
\end{align*}

For $I_2$, we apply \eqref{als0} and use \eqref{fF} to have
\begin{align*}
I_2
&= \int_{\theta t}^t |e^{-(t-\tau)\delta A} (e^{-(t-\tau)(1-\delta)A}A^{1-\varep}f(\tau))|_{\alpha,\sigma}d\tau
\le \int_{\theta t}^t e^{-(t-\tau)\delta} |e^{-(t-\tau)(1-\delta)A}A^{1-\varep}f(\tau)|_{\alpha,\sigma}d\tau\\
&\le \int_{\theta t}^t e^{-\delta (t-\tau)} \Big[\frac{1-\varep}{e(1-\delta)(t-\tau)}\Big]^{1-\varep} |f(\tau)|_{\alpha,\sigma}d\tau 
\le   \Big[\frac{1-\varep}{e(1-\delta)}\Big]^{1-\varep} M F(\theta t) \int_{\theta t}^t \frac{e^{-\delta(t-\tau)}}{(t-\tau)^{1-\varep}}d\tau.
\end{align*}

We estimate the last integral by
\begin{align*}
\int_{\theta t}^t \frac{e^{-\delta(t-\tau)}}{(t-\tau)^{1-\varep}}d\tau&=\int_0^{(1-\theta)t} z^{\varep-1}e^{-\delta z}dz
\le \int_0^{1-\theta} z^{\varep-1}dz
+ (1-\theta)^{\varep-1}\int_{1-\theta}^{(1-\theta)t} e^{-\delta z}dz\\
&\le  \frac{(1-\theta)^\varep}{\varep}+ \frac{(1-\theta)^{\varep-1}}\delta e^{-\delta(1-\theta)}.
\end{align*}

(d) Combining the above calculations, we obtain 
\beq\label{wNrem}
|P_N\big( w(t)-A^{-1}\xi\big)|_{\alpha+1-\varep,\sigma}\le C\big(e^{-(1-\theta')\theta \delta t}+F(\theta'\theta t)\big) \quad \forall t \ge 1,
\eeq
with constant $C$ independent of $N$.
By passing $N\to\infty$ in \eqref{wNrem}, and the fact $A^{-1}\xi\in G_{\alpha+1-\varep,\sigma}$, we obtain  $w(t)\in G_{\alpha+1-\varep,\sigma}$ together with the estimate  \eqref{wremain}.

\textit{Proof of \eqref{wremain3}.} 
Take $\theta\in(a+a_0,1)$, and set $\delta=a_0/(\theta-a)$ and $\theta'=a/\theta$. Then  $\theta-a>a_0>0$, which gives $\delta,\theta'\in(0,1)$. It is also clear that
$\theta\theta'=a$ and $(1-\theta')\theta \delta=(\theta-a)\delta =a_0$.
Therefore, with these values of $\theta,\theta',\delta$, inequality \eqref{wremain3} follows \eqref{wremain}.

(iii) Note, by the monotonicity of $F$, that  the property \eqref{Fa} in fact holds true for all $a>0$, with $D_{2,a}=1$ for all $a\ge 1$.
Then we have 
\beq\label{eafa}
\begin{aligned}
e^{-a_0 t}+F(a t)
&=e^{-k_0 \cdot a_0 t/k_0}+F(a t)\\
& \le D_1 F(a_0 t/k_0) +F(at)\le (D_1 D_{2,a_0/k_0}+D_{2,a})F(t). 
\end{aligned}
\eeq

Combining \eqref{eafa} with \eqref{wremain3}, we obtain  inequality  \eqref{wremain2}.
The proof is complete.
\end{proof}

\subsection{The NSE}\label{LHdecay}

This subsection aims at establishing the large-time estimates for weak solutions of the NSE.
First, we obtain a result for small initial data and force.

\begin{theorem} \label{Ftheo}
Let $F$ be a continuous, decreasing, non-negative function on $[0,\infty)$.
Given  $\alpha\ge1/2$ and numbers $\theta_0,\theta\in (0,1)$ such that 
$\theta_0+\theta<1$.
Then there exist positive numbers $c_k=c_k(\alpha,\theta_0,\theta,F)$, for $k=0,1,2,3$, such that the following holds true.  If
\begin{align}\label{usmall}
|A^\alpha u^0|&\le c_0,\\
\label{fta}
|f(t)|_{\alpha-1/2,\sigma}&\le c_1F(t)\quad\text{a.e. in } (0,\infty)\text{ for some } \sigma\ge 0,
\end{align} 
then there exists a unique regular solution $u(t)$ of \eqref{fctnse} and \eqref{uzero}, which also belongs to  
$C([0,\infty),\mathcal D(A^\alpha))$ and satisfies, for all $t\ge 8\sigma(1-\theta)/(1-\theta-\theta_0)$,
\begin{align}\label{uest}
|u(t)|_{\alpha,\sigma}&\le c_2 (e^{-2\theta_0 t}+F^2(\theta t))^{1/2},\\
\label{intAa}
 \int_t^{t+1} |u(\tau)|_{\alpha+1/2,\sigma}^2d\tau
 &\le  c_3^2 (e^{-2\theta_0 t}+F^2(\theta t)).
 \end{align}
\end{theorem}
\begin{proof}
The proof follows \cite[Theorem 3.1]{CaH1}. The calculations below are formal, but can be made rigorous by using the Galerkin approximations and then pass to the limit.

Let $\theta_*=\theta_0/(1-\theta)\in(\theta_0,1)$ and denote  $t_*=8\sigma/(1-\theta_*)=8\sigma(1-\theta)/(1-\theta-\theta_0)$.

\medskip
(a)  For $\sigma>0$, let $\varphi$ be a $C^\infty$-function on $\R$ such that $\varphi((-\infty,0])=\{0\}$, $\varphi([t_*,\infty))=\{\sigma\}$, and $0< \varphi'< 2\sigma/t_*$ on $(0,t_*)$. 
We derive from \eqref{fctnse} that 
\begin{align}
\ddt (A^{\alpha}e^{\varphi(t) A^{1/2}}u) 
&=\varphi'(t)A^{1/2}A^{\alpha}e^{\varphi(t) A^{1/2}} u+A^{\alpha}e^{\varphi(t) A^{1/2}}\frac{du}{dt}  \notag\\
&=\varphi'(t)A^{\alpha+1/2}e^{\varphi(t) A^{1/2}} u+A^{\alpha}e^{\varphi(t) A^{1/2}}(-Au-B(u,u)+f).\label{daeu}
\end{align}

By taking the inner product in $H$ of  \eqref{daeu} with $A^{\alpha}e^{\varphi(t) A^{1/2}}u(t)$, we obtain 
\begin{align*}
&\frac12\ddt |u|_{\alpha,\varphi(t)}^2   + |A^{1/2}u|_{\alpha,\varphi(t)}^2
=\varphi'(t)\langle A^{\alpha+1/2}e^{\varphi(t) A^{1/2}}u,A^{\alpha}e^{\varphi(t) A^{1/2}} u\rangle\\ 
&\quad -\langle A^{\alpha}e^{\varphi(t) A^{1/2}}B(u,u),A^{\alpha}e^{\varphi(t) A^{1/2}}u\rangle+ \langle A^{\alpha-1/2}e^{\varphi(t) A^{1/2}}f,A^{\alpha+1/2}e^{\varphi(t) A^{1/2}}u\rangle.
\end{align*}

Using the Cauchy-Schwarz inequality, and estimating the second term on the right-hand side  by \eqref{AalphaB}, we get
\beq \label{s0}
\begin{aligned}
\frac12\ddt |u|_{\alpha,\varphi(t)}^2   + |A^{1/2}u|_{\alpha,\varphi(t)}^2
 &\le \varphi'(t) |A^{1/2}u|_{\alpha,\varphi(t)}^2\\
 &\quad +K^\alpha |A^{1/2}u|_{\alpha,\varphi(t)}^2 |u|_{\alpha,\varphi(t)}
+ |f(t)|_{\alpha-1/2,\varphi(t)}|A^{1/2}u|_{\alpha,\varphi(t)}.
\end{aligned}
\eeq 

Using the bound of $\varphi'(t)$ and applying Cauchy's inequality to the last term gives
\begin{align*}
&\frac12\ddt |u|_{\alpha,\varphi(t)}^2   + |A^{1/2}u|_{\alpha,\varphi(t)}^2
\le \frac{2\sigma}{t_*} |A^{1/2}u|_{\alpha,\varphi(t)}^2 \\
&+K^\alpha |u|_{\alpha,\varphi(t)} |A^{1/2}u|_{\alpha,\varphi(t)}^2 
 +\frac{2\sigma}{t_*}|A^{1/2}u|_{\alpha,\varphi(t)}^2+ \frac{t_*}{8\sigma}|f(t)|_{\alpha-1/2,\varphi(t)}^2,
\end{align*}
which, together with the fact $\varphi(t)\le\sigma$, implies
\beqs
\frac12\ddt |u|_{\alpha,\varphi(t)}^2 + \Big(1-\frac{4\sigma}{t_*} -K^\alpha |u|_{\alpha,\varphi(t)}\Big)|A^{1/2}u|_{\alpha,\varphi(t)}^2 \le \frac{t_*}{8\sigma}|f(t)|_{\alpha-1/2,\sigma}^2.
\eeqs
Thus,
 \beq\label{s1}
 \frac12\ddt |A^\alpha u|^2+\Big(1-\frac{1-\theta_*}{2}-K^\alpha|A^\alpha u|\Big)|A^{\alpha+1/2}u|^2\le \frac{1}{1-\theta^*}|A^{\alpha-1/2}f|^2.
 \eeq

(b) For $\sigma=0$, let $\varphi\equiv 0$ on $\R$.  Then the first term in \eqref{s0} vanishes. Applying Cauchy's inequality to the last term of \eqref{s0}:
\beqs
 |f(t)|_{\alpha-1/2,\varphi(t)}|A^{1/2}u|_{\alpha,\varphi(t)}\le \frac{1-\theta_*}2 |A^{1/2}u|_{\alpha,\varphi(t)}^2 + \frac{1}{1-\theta_*}  |f(t)|_{\alpha-1/2,\varphi(t)}^2,
\eeqs
we obtain the same inequality \eqref{s1}.

(c) In the calculations below, we use the following constants
\begin{align*}
c_*&=c_*(\alpha,\theta_0,\theta,F)= \frac{1-\theta_*}{4K^\alpha},\ 
&&\gamma=\gamma(F)=\frac{1}{F(0)+1}\in(0,1],\\
c_0&=c_0(\alpha,\theta_0,\theta,F)=\gamma c_*,\
&&c_1=c_1(\alpha,\theta_0,\theta,F)=\gamma^2 c_* (\theta_*(1-\theta_*))^{1/2},\\
c_2&=c_2(\alpha,\theta_0,\theta,F)=\sqrt{2} \gamma c_*,\
&&c_3=c_3(\alpha,\theta_0,\theta,F)=(1+\theta_*^{-1})^{1/2}\gamma c_* .
\end{align*}

At the initial time, we have  
$$|u(0)|_{\alpha,\varphi(0)}=|A^\alpha u^0|<2c_0\le 2c_*.$$ 

Let $T\in(0,\infty)$. Assume that
\beq\label{uT}
|u(t)|_{\alpha,\varphi(t)}\le 2c_*\quad \forall t\in[0,T).
\eeq

This and the definition of $c_*$ give
\beq\label{Ku}
K^\alpha |u(t)|_{\alpha,\varphi(t)} \le 2c_* K^\alpha=(1-\theta_*)/2 \quad \forall t\in[0,T).
\eeq

For $t\in (0,T)$, we have from \eqref{s1} for both $\sigma>0$ and $\sigma=0$, and \eqref{Ku} that
\beq\label{s2}
\ddt |u|_{\alpha,\varphi(t)}^2 + 2\theta_*|A^{1/2}u|_{\alpha,\varphi(t)}^2 \le 2(1-\theta_*)^{-1}|f(t)|_{\alpha-1/2,\sigma}^2.
\eeq

Applying Gronwall's inequality in \eqref{s2} and using \eqref{usmall}, \eqref{fta} yield, for all $t\in(0,T)$, that
\begin{align*}
|u(t)|_{\alpha,\varphi(t)}^2
&\le e^{-2\theta_*t}|u^0|_{\alpha,0}^2+2(1-\theta_*)^{-1}\int_0^t e^{-2\theta_*(t-\tau)}|f(\tau)|_{\alpha-1/2,\sigma}^2 d\tau\\
&\le e^{-2\theta_0 t}c_0^2+2(1-\theta_*)^{-1}c_1^2\int_0^t e^{-2\theta_*(t-\tau)}F^2(\tau) d\tau.
\end{align*}

For the last integral, applying \eqref{Fi1} with $F:=F^2$, $\sigma:=2\theta_*$ and noting that $(1-\theta)\theta_*=\theta_0$ give 
\beqs 
 \int_0^t e^{-2\theta_* (t-\tau)}F^2(\tau)d\tau
 \le \frac1{2\theta_*}\Big(F^2(0)e^{-2\theta_0 t}+F^2(\theta t)\Big)\quad\forall t\ge 0.
\eeqs

Then  we obtain
\begin{align*}
 |u(t)|_{\alpha,\varphi(t)}^2 
& \le    c_0^2 e^{-2\theta_0 t} + [\theta_*(1-\theta_*)]^{-1}c_1^2\big( \gamma^{-2}e^{-2\theta_0 t}+ F^2(\theta t)\big)\\
&\le   2  c_*^2  \gamma^2(e^{-2\theta_0 t} + F^2(\theta t)).
\end{align*}

This implies 
\beq\label{s4}
|u(t)|_{\alpha,\varphi(t)}
\le  \sqrt{2}  c_*  \gamma (e^{-2\theta_0 t}+F^2(\theta t))^{1/2} \quad\forall t\in[0,T).
\eeq

Letting $t\to T^-$ in \eqref{s4} and using the monotonicity of $F$ give
\beq\label{limTu}
\lim_{t\to T^-}|u(t)|_{\alpha,\varphi(t)}
 \le \sqrt{2}  c_* \gamma (1+F^2(0))^{1/2} < 2c_*.
\eeq

By the standard contradiction argument, it follows \eqref{uT} and \eqref{limTu} that the inequalities \eqref{uT} and \eqref{s4}, in fact, hold true for $T=\infty$.  
Then, due to the fact $\varphi(t)=\sigma$ for all $t\ge t_*$,  inequality \eqref{s4} implies \eqref{uest}.

(d) For $t\ge t_*$, by integrating \eqref{s2} from $t$ to $t+1$, and using estimates \eqref{uest}, \eqref{fta}, we obtain
\begin{align*}
\int_t^{t+1} |A^{1/2}u(\tau)|_{\alpha,\sigma}^2d\tau
& \le \frac{1}{2\theta_*}| u(t)|_{\alpha,\sigma}^2+[\theta_*(1-\theta_*)]^{-1}c_1^2 \int_t^{t+1} F^2(\tau)d\tau\\
&\le c_*^2 \gamma^2  \theta_*^{-1}(e^{-2\theta_0 t}+F^2(\theta t))+ c_*^2 \gamma^2 F^2(\theta t).
\end{align*}
Then inequality \eqref{intAa} follows. The proof is complete.
\end{proof}

\begin{theorem}\label{Fthm2}
Let $F$ be a continuous, decreasing, non-negative function on $[0,\infty)$
 that satisfies
\beq\label{Fzero}
\lim_{t\to\infty} F(t)=0.
\eeq 

Suppose there exist $\sigma\ge 0$, $\alpha\ge 1/2$ such that
\beq\label{falphaonly}
|f(t)|_{\alpha,\sigma}=\mathcal O(F(t)).
\eeq

Let $u(t)$ be a Leray-Hopf weak solution of \eqref{fctnse}.  
Then  there exists  $\hat{T}>0$ 
such that $u(t)$ is a regular solution of \eqref{fctnse} on $[\hat{T},\infty)$, and for any $\varep,\lambda\in(0,1)$, and $a_0,a,\theta_0,\theta\in(0,1)$ with $a_0+a<1$, $\theta_0+\theta<1$, there exists $C>0$ such that
\beq\label{newu}
|u(\hat T + t)|_{\alpha+1-\varep,\sigma}
\le C(e^{-a_0 t}+e^{-2\theta_0 a t}+F^{2\lambda}(\theta a t)+F(at))\quad\forall t\ge 0. 
\eeq

If, in addition, $F$ satisfies \eqref{eF} and \eqref{Fa}, then
 \beq\label{us0}
 |u(\hat{T}+t)|_{\alpha+1-\varep,\sigma} \le CF(t)\quad \forall t\ge 0.
 \eeq 
\end{theorem}

\begin{proof} By \eqref{falphaonly}, there exist $T_1>0$ and $C_1>0$ such that
\beq\label{CF}
|f(t)|_{\alpha,\sigma}\le C_1 F(t)\quad t\ge T_1.
\eeq

We claim the following fact which is weaker than the desired estimate \eqref{us0}.

\medskip
\noindent\textbf{Claim.} 
\textit{ For any $\lambda\in(0,1)$, and $\theta,\theta_0\in(0,1)$ with $\theta+\theta_0<1$, there exists  $\hat{T}\ge T_1$ 
such that $u(t)$ is a regular solution of \eqref{fctnse} on $[\hat{T},\infty)$, and one has for all $t\ge 0$ that
 \beq\label{preus0}
 |u(\hat{T}+t)|_{\alpha+1/2,\sigma} \le C(e^{-2\theta_0 t}+ F^{2\lambda}(\theta t))^{1/2},
 \eeq
for some positive constant $C$.}

Accepting this Claim at the moment, we prove \eqref{us0}. 
Rewrite the NSE \eqref{fctnse} as the linearized NSE:
\beq\label{ulin} u_t+Au=\tilde f(t)\eqdef -B(u(t),u(t))+f(t) .
\eeq

Then from \eqref{preus0}  and \eqref{AalphaB} we obtain
for $t$ large,
\beqs
|B(u(\hat T+t),u(\hat T+t))|_{\alpha,\sigma} \le K^\alpha |u(\hat T+t)|_{\alpha+1/2,\sigma}^2 \le C_2(e^{-2\theta_0 t}+F^{2\lambda}(\theta t))
\eeqs 
for some positive constant $C_2$.
From this and \eqref{CF}, we have, for $t\ge 0$,
\beq\label{tiF}
|\tilde f(\hat T+t)|_{\alpha,\sigma}\le C_2(e^{-2\theta_0 t}+F^{2\lambda}(\theta t))+C_1 F(\hat T+t )\le C_3 \tilde F(t),
\eeq
where $C_3=C_1+C_2$ and $\tilde F(t)=e^{-2\theta_0 t}+F^{2\lambda}(\theta t)+F(t)$.

By \eqref{ulin} and \eqref{tiF}, we apply part (iii) of Theorem \ref{Fode} with $w(t):=u(\hat T+t)$, $f(t):=\tilde f(\hat T+t)$, $\xi=0$,
$M=C_3$, $F(t):=\tilde F(t)$ to obtain from \eqref{wremain3}, with $t:=t+1$, that
\begin{align*}
|u(\hat T + t + 1)|_{\alpha+1-\varep,\sigma}
&\le C_4(e^{-a_0(t+1)}+\tilde F(a(t+1)))\\
&\le C_4(e^{-a_0 t}+e^{-2\theta_0 a t}+F^{2\lambda}(\theta a t)+F(at)). 
\end{align*}
for all $t \ge 0$ and some constant $C_4>0$.
By re-denoting $\hat T:=\hat T+1$, we obtain \eqref{newu}.

Now, assume \eqref{eF} and \eqref{Fa}. Taking $\lambda=1/2$, $a_0=\theta_0\in(0,1/2)$ and $a=1/2$, we obtain from \eqref{newu}
 \beq\label{us1}
 |u(\hat{T}+t)|_{\alpha+1-\varep,\sigma} \le C(e^{-\theta_0 t}+F(\theta t/2))\quad \forall t\ge 0. 
 \eeq

Similarly to proving \eqref{eafa}, we obtain inequality \eqref{us0} from \eqref{us1}.

\medskip
The rest of this proof is to prove the Claim.

\medskip
(a) By Assumption \ref{A1}, there exists  $C_0>0$ such that
\beq\label{fkappa}
|f(t)|\le C_0, \quad\text{a.e. in } (0,T_1).
\eeq

On the one hand, using  \eqref{iniener}, \eqref{CF}, \eqref{fkappa} we have, for all $t\ge T_1$, that
\begin{align*}
|u(t)|^2
&\le e^{-t}|u_0|^2 + C_0^2 \int_0^{T_1} e^{-(t-\tau)}d\tau + C_1^2 \int_{T_1}^t e^{-(t-\tau)}F^2(\tau)d\tau\\
&\le e^{-t}|u_0|^2 + C_0^2 e^{-t}e^{T_1} + C_1^2 \int_0^t e^{-(t-\tau)}F^2(\tau)d\tau.
\end{align*}

To  estimate the last integral, we apply inequality \eqref{Fi1} with $\sigma:=1$, $\theta:=1/2$, $F:=F^2$, hence, obtain 
\beq\label{uenerM}
|u(t)|^2\le  e^{-t}(|u_0|^2+C_0^2e^{T_1}) + C_1^2 (F^2(0)e^{-t/2}+F^2(t/2))\quad\forall t\ge T_1.
\eeq

On the other hand, we estimate in \eqref{Lenergy}
\beqs
|\inprod{f(\tau),u(\tau)}|\le \frac12 |u(\tau)|^2+\frac12 |f(\tau)|^2 \le \frac12 \|u(\tau)\|^2+\frac12 |f(\tau)|^2.
\eeqs
Hence, we obtain 
	\beq\label{intt0}
		|u(t)|^2+\int_{t_0}^t \|u(\tau)\|^2d\tau\le |u(t_0)|^2+\int_{t_0}^t |f(\tau)|^2\ d\tau
	\eeq
for all $t_0\in \goodt$ and  $t\ge t_0$.

Let $t_0\in\goodt\cap [T_1,\infty)$. Setting $t=t_0+1$ in \eqref{intt0}, using \eqref{uenerM} to estimate $|u(t_0)|^2$, and \eqref{CF} to estimate $|f(\tau)|$, we derive
\beqs
\int_{t_0}^{t_0+1}\|u(\tau)\|^2d\tau 
\le e^{-t_0}(|u_0|^2+C_0^2 e^{T_1}) + C_1^2 (F^2(0)e^{-t_0/2}+F^2(t_0/2)) + C_1^2 F^2(t_0),
\eeqs
thus,
\beq\label{uVest}
\int_{t_0}^{t_0+1}\|u(\tau)\|^2d\tau 
\le e^{-t_0/2}(|u_0|^2 +C_0^2 e^{T_1}+ C_1^2 F^2(0)) + 2C_1^2F^2(t_0/2)).
\eeq 

Let $t\ge T_1$ be arbitrary now. There exists a sequence $\{t_n\}_{n=1}^\infty$ in $\goodt\cap (T_1,\infty)$ such that  $\lim_{n\to\infty}t_n=t$. Then \eqref{uVest} holds for $t_0=t_n$, and letting $n\to\infty$ gives
 \beq\label{tt1}
\int_t^{t+1}\|u(\tau)\|^2d\tau 
\le M_t\eqdef e^{-t/2}(|u_0|^2 + C_0^2 e^{T_1}+C_1^2F^2(0))+2C_1^2F^2(t/2))\quad\forall t\ge T_1.
\eeq

Note that the quantity $M_t$ in \eqref{tt1} is decreasing in $t$, and goes to zero as $t$ tends to infinity.

\medskip
(b) Consider $\sigma>0$. Let  $\lambda\in(0,1)$.
For $T>0$,  we write
\beq\label{Fsplit}
F(t+T)=F^{1-\lambda}(t+T)F^{\lambda}(t+T)\le F(T)^{1-\lambda}F^\lambda(t).
\eeq

Choose $T_2>T_1$ such that 
\beqs
M_{T_2}< c_0(1/2,F^\lambda)/2\text{ and }  F(T_2)^{1-\lambda}\le c_1(1/2,F^\lambda)/C_1.
\eeqs

By applying inequality \eqref{tt1} to $t=T_2$, there exists $t_0\in\goodt \cap (T_2,T_2+1)$ such that 
\beqs
|A^{1/2}u(t_0)|\le 2M_{t_0}\le 2 M_{T_2} <c_0(1/2,F^\lambda).
\eeqs

Moreover, for $t\ge 0$, by \eqref{CF} and \eqref{Fsplit},
\begin{align}
|f(t_0+t)|_{0,\sigma} 
&\le C_1 F(t_0+t)\le C_1 F(t_0)^{1-\lambda}F^\lambda(t)\le C_1 F(T_2)^{1-\lambda}F^\lambda(t)\notag\\
&\le c_1(1/2,F^\lambda) F^\lambda(t). \label{ft0}
\end{align}

Applying Theorem \ref{Ftheo} to the unique regular solution $u(t):= u(t_0+t)$, force  $f(t):= f(t_0+t)$ with parameters $\alpha=1/2$ and $F(t):=F^\lambda(t)$, we obtain from \eqref{uest} that 
\beq\label{ut1}
|u(t_0+t)|_{1/2,\sigma}\le c_2(1/2,F^\lambda) (e^{-\theta_0 t}+F^\lambda(\theta t))\quad \forall t\ge t_*,
\eeq
where $t_*$ is a non-negative number.
Then by \eqref{als}, we have for all $t\ge t_*$ that
\beqs
|A^{\alpha+1/2} u(t_0+t)|
\le  d_0(2\alpha+1,\sigma) |e^{\sigma A^{1/2}} u(t_0+t)|
\le d_0(2\alpha+1,\sigma)   |u(t_0+t)|_{1/2,\sigma},
\eeqs
and, hence, thanks to \eqref{ut1},
\beq\label{Aaut0}
\lim_{t\to\infty} |A^{\alpha+1/2} u(t_0+t)|=0.
\eeq

Using \eqref{Aaut0}, and similar to \eqref{ft0} with the norm $|\cdot|_{\alpha,\sigma}$ replacing $|\cdot|_{0,\sigma}$, we deduce that there is  $T\in\goodt\cap(t_0+t_*,\infty)$ so that
\begin{align}\label{Aut}
|A^{\alpha+1/2} u(T)|&\le c_0(\alpha+1/2,F^\lambda),\\
\label{ftt}
|f(T+t)|_{\alpha,\sigma}&\le c_1(\alpha+1/2,F^\lambda)F^\lambda(t) \quad \forall t\ge 0.
\end{align}

\medskip
(c) We will establish \eqref{Aut} and \eqref{ftt} when $\sigma=0$. 
First, we observe the following: if $j\in \N$ such that  $j\le 2\alpha+1$ and 
\beq\label{intAj}
\lim_{t\to\infty}\int_t^{t+1}|A^{j/2}u(\tau)|^2d\tau=0, 
\eeq
then
\beq\label{intAp1}
\lim_{t\to\infty}\int_t^{t+1}|A^{(j+1)/2}u(\tau)|^2d\tau=0.
\eeq

Indeed, since  $(j-1)/2 \le\alpha$, and thanks to \eqref{falphaonly}, we have 
\beq\label{fj}
|A^\frac{j-1}2f(t)|=\mathcal O(F(t)).
\eeq
By \eqref{intAj} and \eqref{fj}, we obtain, similar to \eqref{Aut} and \eqref{ftt} that there exists $T_3\in\goodt\cap[T_1,\infty)$ so that
\beqs
|A^{j/2} u(T_3)|\le c_0(j/2,F^\lambda),
\eeqs
\beqs
|A^{j/2-1/2}f(T_3+t)|\le c_1(j/2,F^\lambda)F^\lambda(t) \quad \forall t\ge 0.
\eeqs

Applying Theorem \ref{Ftheo}  to $u(t):=u(T_3+t)$, $f(t):=f(T_3+\cdot)$, $F(t):=F^\lambda(t)$, $\alpha:=j/2$, $\sigma:=0$,  we obtain from \eqref{intAa} that
\begin{align*}
\int_t^{t+1}|A^{(j+1)/2}u(\tau)|^2d\tau
&=\int_{t-T_3}^{t+1-T_3}|A^{(j+1)/2}u(T_3+\tau)|^2d\tau\\
&=\mathcal O(e^{-2\theta_0(t-T_3)}+F^{2\lambda}(\theta(t-T_3))),
\end{align*}
which proves \eqref{intAp1}, thanks to \eqref{Fzero}. 

\medskip
Now, let $m$ be a non-negative integer such that
$2\alpha \le m<2\alpha+1$. 

Note that $m\ge 1$, and, because of \eqref{tt1}, condition \eqref{intAj} holds true for $j=1$.
Hence we obtain \eqref{intAp1} with $j=1$, which is \eqref{intAj} for $j=2$.
We apply the arguments recursively for $j=1,2,\ldots,m$, and obtain, when $j=m$, from \eqref{intAp1} that
\beqs
\lim_{t\to\infty} \int_t^{t+1}|A^{(m+1)/2}u(\tau)|^2d\tau=0.
\eeqs
Since $\alpha\le m/2$, it follows that
\beq\label{intAm}
\lim_{t\to\infty}\int_t^{t+1}|A^{\alpha+1/2}u(\tau)|^2d\tau=0.
\eeq

By \eqref{intAm}, \eqref{CF}, \eqref{Fsplit} and \eqref{Fzero}, there exists $T\in\goodt\cap[T_1,\infty)$ so that \eqref{Aut} and \eqref{ftt} similarly hold true (for this case of $\sigma=0$.)

\medskip
(d) With $T\in\goodt\cap[T_1,\infty)$ in (b) and (c), we apply Theorem \ref{Ftheo} to the unique regular solution $u(t):=u(T+t)$, $f(t):=f(T+t)$, $F(t):=F^\lambda(t)$, $\alpha:=\alpha+1/2$, and obtain that there is $t_*\ge 0$ such that, following \eqref{uest} with $t:=t+t_*$,
\begin{align*}
|u(T+t_* +t)|_{\alpha+1/2,\sigma}
&\le c_2(\alpha+1/2,\theta_0,\theta,F^\lambda)\Big(e^{-2\theta_0(t_*+ t)}+ F^{2\lambda}(\theta(t_*+t))\Big)^{1/2}\\
&\le C(e^{-2\theta_0 t}+ F^{2\lambda}(\theta t))^{1/2}
\end{align*} 
for all $t\ge 0$. By setting  $\hat T =T+t_*$, this estimate implies \eqref{preus0}. The proof is complete.
\end{proof}

\begin{remark}
 Because the constants $c_0$ and $c_1$ in \eqref{Aut}, \eqref{ftt} can be small, we could not prove \eqref{preus0}  for $\lambda=1$ directly. Rather, we use  \eqref{ulin} and the estimate \eqref{wremain2} in Theorem \ref{Fode} for the linearized  NSE to improve \eqref{preus0} to \eqref{us0}.
\end{remark}

\section{General asymptotic expansions}\label{syssec}

Now, we introduce a very general definition of an asymptotic expansion in a normed space with respect to a system of time-decaying functions.

\begin{definition}\label{sys0}
Let $(\psi_n)_{n=1}^\infty$ be a sequence of non-negative functions defined on $[T_*,\infty)$ for some $T_*\in\R$ that satisfies the following two conditions:

\begin{enumerate}[label=\rm (\alph*)]
\item For each $n\in \N$,
\beq\label{limzero}
\lim_{t\to\infty}\psi_n(t)=0.
\eeq

\item For $n>m$,
\beq\label{noquot}
\psi_n(t)=o(\psi_m(t)).
\eeq
\end{enumerate}

Let $(X, \|\cdot\|)$ be a  normed space, and $g$ be a function from $[T_*,\infty)$ to $X$.
We say $g$ has an asymptotic expansion (implicitly as $t\to\infty$)
\beq\label{g0}
g(t)\sim \sum_{n=1}^\infty \xi_n \psi_n(t) \text{ in }X,
\eeq
where $\xi_n\in X$ for all $n\in\N$, 
if, for any $N\in\N$,
\beq\label{g1}
\|g(t)-\sum_{n=1}^N \xi_n\psi_n(t)\|=o(\psi_N(t)).
\eeq
\end{definition}

Obviously, if $g(t)=\sum_{n=1}^N \xi_n \psi_n(t)$ for some $N\in\N$, then
$g(t)\sim \sum_{n=1}^\infty \xi_n \psi_n(t)$ where $\xi_n=0$ for $n>N$.
In case of the infinite sum, the convergent series 
\beq\label{gequal}
g(t)=\sum_{n=1}^\infty \xi_n \psi_n(t)
\eeq 
does not necessarily imply the expansion \eqref{g0}.
We refer to Appendix \ref{apA} for some criteria for both \eqref{g0} and \eqref{gequal} to hold, with the infinite sum not reduced to a finite one.

Note that the expansion \eqref{g0} does not determine the function $g$. Indeed, if $h:[T_*,\infty)\to X$ is a function that satisfies
$\|h(t)\|=o(\psi_n(t))$  for all $n\in \N$, then both $g$ and $g+h$ have the same expansion on the right-hand side of \eqref{g0}.
The converse is considered in the next proposition.

\begin{proposition}\label{Prop62} 
Let $(\psi_n)_{n=1}^\infty$, $(X,\|\cdot\|)$ and $g$ be as in Definition \ref{sys0}. 
Suppose, for each $n\in\N$, that the function $\psi_n$ is not identically zero on $[T,\infty)$ for all $T\ge T_*$. 
Then the asymptotic expansion \eqref{g0}, if exists, is unique.
\end{proposition}
\begin{proof}
Suppose $g(t)$ has two expansions
\beq\label{gPE}
g(t)\sim\sum_{n=1}^\infty \phi_n \psi_n(t)\text{ and }g(t)\sim\sum_{n=1}^\infty \xi_n \psi_n(t).
\eeq

We will prove by induction that $\phi_n=\xi_n$ for all $n\in\N$.

One has from the triangle inequality and each expansion in \eqref{gPE} that
\beq\label{pmx}
\|(\phi_1-\xi_1)\psi_1(t)\|\le \|\phi_1\psi_1(t)-g(t)\|+\|g(t)-\xi_1\psi_1(t)\|=o(\psi_1(t)).
\eeq

Since $\psi_1$ is asymptotically non-trivial then one can verify from \eqref{pmx} that $\phi_1=\xi_1$. 

Let $N\in \N$ and assume $\phi_n=\xi_n$ for $n=1,2,\ldots,N$.
Then
\begin{align*}
&\|(\phi_{N+1}-\xi_{N+1})\psi_{N+1}(t)\|
=\|\sum_{n=1}^{N+1}(\phi_n-\xi_n)\psi_n(t)\|\\
&\le \|\sum_{n=1}^{N+1}\phi_n\psi_n(t)-g(t)\|+\|g(t)-\sum_{n=1}^{N+1}\xi_n\psi_n(t)\|
=o(\psi_{N+1}(t)). 
\end{align*}
Hence, $\phi_{N+1}=\xi_{N+1}$. By the induction principle, $\phi_n=\xi_n$ for all $n\in\N$.
\end{proof}

Note  that the rate of convergence, as $t\to\infty$, in \eqref{g1}, in fact, can be related to the next term $\psi_{N+1}(t)$.
Indeed,
\beqs
\|g(t)-\sum_{n=1}^N \xi_n\psi_n(t)\|\le \|g(t)-\sum_{n=1}^{N+1} \xi_n\psi_n(t)\| +\psi_{N+1}(t)\|\xi_{N+1}\|.
\eeqs
Hence, we can replace equivalently \eqref{g1} by 
\beq\label{g2}
\|g(t)-\sum_{n=1}^N \xi_n\psi_n(t)\|=O(\psi_{N+1}(t)).
\eeq

The equivalence of \eqref{g1} and \eqref{g2} is essentially due to the infinite sum in \eqref{g0}.
If the sum is finite, this is no more the case.
Moreover, for general $\psi_n$'s, the relation \eqref{noquot} is not informative enough to work with. 

These prompt us to have the following more specific definition.

\begin{definition}\label{premsys}
Let $\Psi=(\psi_\lambda)_{\lambda>0}$ be a system of functions that satisfies the following two conditions.
\begin{enumerate}[label=\rm (\alph*)]
 \item There exists $T_*\ge 0$ such that, for each $\lambda>0$, $\psi_\lambda$ is a positive function defined on $[T_*,\infty)$, and 
 \beq\label{lim0}
 \lim_{t\to\infty}\psi_\lambda(t)=0.
 \eeq

 \item  For any $\lambda>\mu$, there exists $\eta>0$ such that
 \beq \label{quot0}
  \psi_\lambda(t)=\mathcal O(\psi_\mu(t)\psi_\eta(t)).
 \eeq
\end{enumerate}

Let $(X, \|\cdot\|)$ be a real normed space, and $g$ be a function from $(0,\infty)$ to $X$.
 The function $g$ is said to have the asymptotic expansion
	\beq \label{psexpand}
g(t) \simP \sum_{n=1}^\infty \xi_n \psi_{\lambda_n}(t) \text{ in } X,
\eeq
where $\xi_n\in X$ for all $n\in\N$, and $(\lambda_n)_{n=1}^\infty$ is a strictly increasing, divergent sequence of positive numbers, if it holds, for any $N\geq1$, that 
 there exists $\varep>0$ such that
\beq \label{lamrate}
\Big\|g(t)- \sum_{n=1}^N \xi_n \psi_{\lambda_n}(t)\Big\|=\mathcal O(\psi_{\lambda_N}(t)\psi_\varep(t)).
\eeq
\end{definition}

\medskip
We have the following remarks on Definition \ref{premsys}.
\begin{enumerate}[label=\rm (\alph*)]
\item If   $\lambda>\mu$, it follows \eqref{quot0} and \eqref{lim0} that
\beq \label{quot1}
\psi_\lambda(t)=o(\psi_\mu(t)).
\eeq

\item If a function $g$ has an expansion \eqref{psexpand}, then $g(t) \sim \sum_{n=1}^\infty \xi_n \psi_{\lambda_n}(t)$  in $X$ in the sense of Definition \ref{sys0}. 

\item Thanks to (b) and Proposition \ref{Prop62}, the $\xi_n$'s in \eqref{psexpand} are unique.
Similarly, following the proof of Proposition \ref{Prop62}, we also have the uniqueness of $\xi_n$'s  in \eqref{finiteps}.

\item The main difference between Definition \ref{sys0} and Definition \ref{premsys}  is the specific decaying rate $\psi_\eta(t)$ on the right-hand side of \eqref{quot0}, in contrast with the non-specific one in \eqref{noquot}.
In the proofs, this crucially allows comparisons and estimates for different quantities.
\end{enumerate}

We have the following special cases for the expansion \eqref{psexpand}.

\begin{enumerate}[label=\rm (\roman*)]

\item Assume \eqref{psexpand}. If there exists $N\in\N$, such that 
\beq\label{xicut} \xi_n=0\text{ for all }n>N,
\eeq 
then it holds for all $\lambda>0$ that 
\beq \label{finiterate}
\Big\|g(t)- \sum_{n=1}^N \xi_n \psi_{\lambda_n}(t)\Big\|=\mathcal O(\psi_\lambda(t)).
\eeq

\item Assume there exist $N\in\N$, $\xi_n\in X$ for $1\le n\le N$, and $\lambda_n$'s, for $1\le n\le N$, are positive numbers, strictly increasing in $n$ such that
\eqref{finiterate} holds for all $\lambda>0$.
We extend $\xi_n\in X$ for $1\le n\le N$ to a sequence $(\xi_n)_{n=1}^\infty$ with \eqref{xicut},
and extend $\lambda_n$ for $1\le n\le N$ to any sequence $(\lambda_n)_{n=1}^\infty$ that is a strictly increasing and divergent.
Then one can verify that \eqref{psexpand} holds true.

Therefore, we say in cases (i) and (ii) that the function $g$ has the asymptotic expansion
\beq \label{finiteps}
g(t) \simP \sum_{n=1}^N \xi_n \psi_{\lambda_n}(t) \text{ in } X.
\eeq

\item If $\xi_n=0$ for all $n\in N$ in \eqref{psexpand}, then 
\beq \label{zeroexp}
\|g(t)\|=\mathcal O(\psi_\lambda(t))
\eeq
for all $\lambda>0$.

\item Assume \eqref{zeroexp} holds for all $\lambda>0$. Let $\xi_n=0$ for all $n\in\N$. Let $(\lambda_n)_{n=1}^\infty$ be any strictly increasing, divergent sequence of positive numbers.
Then we have \eqref{psexpand}.

Therefore, we say in cases (iii) and (iv) that the function $g$ has the asymptotic expansion
\beqs
g(t) \simP 0  \text{ in } X.
\eeqs

\item For $N=0$, we conveniently set the sum $\sum_{n=1}^N \xi_n \psi_{\lambda_n}(t)$ to be zero in \eqref{finiteps}, and see that the condition \eqref{finiterate}  is, in fact, \eqref{zeroexp}.
Thus the expression 
\beqs 
g(t) \simP \sum_{n=1}^0 \xi_n \psi_{\lambda_n}(t)\quad \text{ will mean  }\quad g(t) \simP 0.
\eeqs 

\item If a function $g$ has an asymptotic expansion
\beqs
g(t) \simP \sum_{n=1}^N \xi_n \psi_{\lambda_n}(t) \text{ in } X,
\eeqs
for $N\in\N\cup\{0,\infty\}$, then by remark (c) above, this expansion is unique for $g$.
\end{enumerate}

\medskip
For solutions of ODEs or PDEs the linear and nonlinear structures of the equations will impose more conditions on the system. We consider below the ones that are appropriate to our current study of the NSE.

\begin{condition}\label{sysdef} The system $\Psi=(\psi_\lambda)_{\lambda>0}$ satisfies {\rm (a)} and {\rm (b)} in Definition \ref{premsys} and the following. 
\begin{enumerate}[label=\rm (\roman*)]
 \item For any $\lambda,\mu>0$, there exist $\gamma > \max\{\lambda,\mu\}$ and a nonzero constant $d_{\lambda,\mu}$ such that 
 \beq\label{triple} \psi_\lambda\psi_\mu=d_{\lambda,\mu}\psi_\gamma.
 \eeq 
 
 \item For each $\lambda>0$, the function $\psi_\lambda$ is continuous and differentiable on $[T_*,\infty)$, and  its derivative $\psi_\lambda'$ has an expansion in the sense of Definition \ref{premsys}
 \beq\label{ppex}
 \psi_\lambda'(t)\simP \sum_{k=1}^{N_\lambda} c_{\lambda,k}\psi_{\lambda^\vee(k)}(t) \text{ in }\R,
 \eeq
  where $N_\lambda\in\N\cup \{0,\infty\}$, all $c_{\lambda,k}$ are constants,  all $\lambda^\vee(k)>\lambda$,
  and, for each $\lambda>0$,  $\lambda^\vee(k)$'s are strictly increasing in $k$.
\end{enumerate}
\end{condition}

 The following remarks on Condition \ref{sysdef} are in order.
\begin{enumerate}[label=\rm (\alph*)]
\item  By \eqref{quot1}, the numbers $\gamma$ and $d_{\lambda,\mu}$ in \eqref{triple} are unique. 


\item By  \eqref{triple}, we have the reverse of \eqref{quot0} in the following sense:
 \beq\label{tq}
 \psi_\mu(t)\psi_\eta(t)=\mathcal O(\psi_       \lambda(t))\text{ for some }\lambda>\mu.
 \eeq

\item Thanks to \eqref{quot0} and \eqref{tq}, the condition \eqref{lamrate} is equivalent to 
\beqs 
\Big\|g(t)- \sum_{n=1}^N \xi_n \psi_{\lambda_n}(t)\Big\|=\mathcal O(\psi_\lambda(t)),
\eeqs
for some $\lambda>\lambda_N$.
 
\item Expansion \eqref{ppex} of $\psi_\lambda'$, in all cases of $N_\lambda$, and property \eqref{quot0}  imply that there exists $\eta>0$ such that
\beq \label{psi-prime}
|\psi_{\lambda}'(t)| =  \mathcal O(\psi_{\lambda}(t)\psi_{\eta}(t)).
\eeq

\item If, instead of \eqref{ppex}, $\psi_\lambda'=c_\lambda \psi_\lambda$ for all $\lambda$, then $\psi_\lambda$'s are exponential functions. This case was studied in \cite{HM2}. For other examples of \eqref{ppex}, see subsections \ref{ss3} and \ref{ss4} below.

\end{enumerate}

\begin{notation}
Denote $\gamma$ in \eqref{triple} by $\lambda\wedge \mu$, which is uniquely determined thanks to remark (a) above. 
\end{notation}

For the current study, we focus on decaying functions that are larger than the exponentially decaying ones, hence we impose more specific conditions.

\begin{condition}\label{weaksys} The system $\Psi=(\psi_\lambda)_{\lambda>0}$  satisfies {\rm(a)}, {\rm(b)} of Definition \ref{premsys}, and the following.
\begin{enumerate}[label=\rm (\roman*)]
 \item For each $\lambda>0$, the function $\psi_\lambda$ is decreasing (in $t$).

\item If $\lambda,\alpha>0$ then 
 \beq\label{expops} e^{-\alpha t}=o(\psi_\lambda(t)).
 \eeq

 \item For any number $a\in(0,1)$,
 \beq\label{shiftlim}
\psi_\lambda(at)=\mathcal O(\psi_\lambda(t)).
\eeq 
\end{enumerate}
\end{condition}

The followings are direct consequences of Condition \ref{weaksys}.
\begin{enumerate}[label=\rm (\alph*)]
\item By \eqref{expops}, for any $\alpha,\lambda>0$, there exists a positive constant $C_{\alpha,\lambda}$ such that
\beqs
e^{-\alpha (t+T_*)}\le C_{\alpha,\lambda} \psi_\lambda(t+T_*)\quad\forall t\ge 0,
\eeqs
hence, by denoting $D_1(\lambda,\alpha)=e^{\alpha T_*} C_{\lambda,\alpha} $, we have 
\beq\label{mxlam}
e^{-\alpha t}\le D_1(\lambda,\alpha) \psi_\lambda(t+T_*).
\eeq

\item For $a\in(0,1)$, we have from (i) that $\psi_\lambda(t)=\mathcal O(\psi_\lambda(at))$.
Thus, the condition \eqref{shiftlim}, in fact, is equivalent to 
\beqs
\psi_\lambda(at)\Oeq \psi_\lambda(t).
\eeqs

\item  Property \eqref{shiftlim} and the decrease of $\psi_\lambda(t)$ in $t$ imply, for $a\in(0,1)$, that
\beqs 
\psi_\lambda(at)\le D_2(a,\lambda) \psi_\lambda(t)\quad \forall t\ge T_*/a,
\eeqs
where $D_2(a,\lambda)$ is a constant in $[1,\infty)$.
Consequently, for $a\in (0,1)$ and $t\ge 0$, 
\beq \label{ps7}
\psi_\lambda(at+T_*)=\psi_\lambda(a(t+T_*/a))\le D_2(a,\lambda) \psi_\lambda(t+T_*/a)
\le D_2(a,\lambda) \psi_\lambda(t+T_*).
\eeq
Then by the decrease of $\psi_\lambda$ in $t$, we have 
\beq \label{ps5}
\psi_\lambda(at+T_*)\le D_2(a,\lambda) \psi_\lambda(t)\quad\forall t\ge T_*.
\eeq

In particular, for any $T\ge 0$ and $t\ge 2(T_*+T)$, we have $t-T\ge t/2+T_*$, then by \eqref{ps5},
\beqs 
\psi_\lambda(t-T)\le \psi_\lambda(t/2+T_*)\le  D_3(\lambda) \psi_\lambda(t),
\text{ where } D_3(\lambda)=D_2(1/2,\lambda). 
\eeqs
Combining this with the boundedness of $\psi_\lambda(t-T)/\psi_\lambda(t)$ for small $t\in[T_*+T,2(T_*+T)]$ we obtain
\beqs 
\psi_\lambda(t-T)\le D_4(\lambda,T) \psi_\lambda(t)\quad\forall t\ge T_*+T,
\eeqs
which yields
\beqs 
\psi_\lambda(t)\le D_4(\lambda,T) \psi_\lambda(t+T)\quad\forall t\ge T_*,
\eeqs
for some positive constant $D_4(\lambda,T)$. Consequently, for any $T\in\R$,
\beq\label{ps2}
\psi_\lambda(t)\Oeq \psi_\lambda(t+T).
\eeq
\end{enumerate}

In applications to the NSE, suppose that the force $f$ has an expansion containing the terms $\psi_{\gamma_n}$'s for some numeric sequence $(\gamma_n)_{n=1}^\infty$. 
Then the operators in the NSE require that a solution $u(t)$, in case of having an expansion itself, may need many more terms in addition to $\psi_{\gamma_n}$'s. 
We describe below a general principle to find those other terms.

For any $x\in (0,\infty)$, define the set
\beqs
G_x=\begin{cases}
     \emptyset,&\text{ if } N_x=0,\\
     \{x^\vee(k):1\le k\le N_x\},&\text{ if }N_x\in \N,\\
     \{x^\vee(k):k\in\N\},&\text{ if }N_x=\infty.
    \end{cases}
\eeqs

A non-empty subset $S$ of $(0,\infty)$ is said to preserve the operation $\vee$ if
\beq \label{S1}
\forall x\in S: G_x\subset S.
\eeq 

Similarly, $S$ is said to preserve the operation $\wedge$ if
\beq \label{S2}
\forall x,y\in S: x\wedge y\in S.
\eeq

\begin{lemma}\label{Slem}
Let $S$ be any non-empty subset of $(0,\infty)$.
\begin{enumerate}[label=\rm (\roman*)]
 \item There exists a smallest set $S_*\subset(0,\infty)$ that contains $S$, and preserves the operations $\vee$ and $\wedge$.
 \item In fact, $S_*=S^*$, where $S^*$ is constructed explicitly in \eqref{Sdef1} below.
\end{enumerate}
\end{lemma}
\begin{proof}
(i) For any non-empty subset $M$ of $(0,\infty)$, we denote
 \begin{align*}
 M^{\wedge}&= \{x\wedge y:x,y\in M\},\\
 M^{\vee}&=\bigcup_{x\in M} G_x. 
 \end{align*}

(a) Let $S_0=S$. We define recursively the sets $S_n$, for $n\in\N$, by
\beqs
S_{2k+1}=S_{2k} \cup S_{2k}^{\vee}\text{ and } S_{2k+2}=S_{2k+1} \cup S_{2k+1}^{\wedge}\text{ for }k\ge 0.
\eeqs

Define 
\beq\label{Sdef1} S^*=\bigcup_{n=0}^\infty S_{n}.
\eeq 

We obviously have
\beq\label{Skseq} S_{2k}\subset S_{2k+1}\subset S_{2k+2}\subset S_{2k+3}\quad\forall k\ge 0.
\eeq

It follows \eqref{Skseq} that $(S_n)$, $(S_{2n})$ and $(S_{2n+1})$ are increasing sequences, and, hence,
\beq\label{Ssdef}
S^*=\bigcup_{k=0}^\infty S_{2k}=\bigcup_{k=0}^\infty S_{2k+1}.
\eeq

Clearly, $S=S_0\subset S^*$.  Next, we prove  $S^*$ preserves the operations $\vee$ and $\wedge$.

Let $x\in S^*$. By \eqref{Ssdef},  $x\in S_{2k}$ for some $k\ge 0$. Then obviously by definition $G_x \subset S_{2k}^{\vee} \subset S_{2k+1}\subset S^*$. Thus, $G_x\subset S^*$.

For $x,y\in S^*$, then by \eqref{Ssdef}, $x\in S_{2k+1}$, $y\in S_{2m+1}$  for some $k,m\ge 0$. Assume $k\ge m$, then $y\in S_{2k+1}$ by \eqref{Skseq}. This implies $x\wedge y\in S_{2k+2}\subset S^*$.

(b) Let  $\mathcal C$ be the collection  of sets $M$ that contain $S$ and  preserve the operations $\vee$ and $\wedge$.
Because $S^*\in\mathcal C$, then the collection $\mathcal C$ is non-empty.

Let $S_*$ be the intersections of all the elements in $\mathcal C$.
Then $S_* \subset S^*$. Let $M \in \mathcal C$, properties \eqref{S1} and \eqref{S2} for $S:=M$ clearly imply
\beq \label{SSS}
M^{\vee} \subset M \text{ and }M^{\wedge} \subset M.
\eeq
Thus, 
\beqs
(S_*)^\vee\subset M^{\vee} \subset M \quad \text{ and } (S_*)^{\wedge}\subset M^{\wedge} \subset M.
\eeqs
It follows that 
\beqs
(S_*)^\vee\subset \bigcap_{M\in C}M=S_* \quad \text{ and } (S_*)^{\wedge}\subset \bigcap_{M\in C}M=S_*.
\eeqs
Therefore, $S_*\in \mathcal C$. By its definition, $S_*$ is the smallest set in $\mathcal C$.

(ii) We prove $S_*=S^*$. It suffices to show $S^* \subset S_*$.

Let  $M$ be an arbitrary element in $\mathcal C$. We shall show that $S^* \subset M$. First, we see that 
$S_0 \subset M$. By \eqref{SSS},
\beqs
S_1=S_0 \cup S_0^{\vee} \subset M \cup M^{\vee}=M,
\eeqs
and then,
\beqs
S_2=S_1 \cup S_1^{\wedge} \subset M \cup M^{\wedge}=M.
\eeqs
By induction, we can prove similarly that $S_k \subset M$ for all $k$.  Therefore $S^*=\bigcup_{k=0}^\infty S_{k} \subset M$. 

Then $S^*\subset \bigcap_{M\in \mathcal C} M=S_*$. This completes the proof of the lemma.
\end{proof}

\begin{notation}
 We denote the set $S_*$ in Lemma \ref{Slem} by $\mathcal G_\Psi(S)$.
\end{notation}

\section{Asymptotic expansions in a continuum system}\label{expsec}

Let $\Psi=(\psi_\lambda)_{\lambda>0}$ be a system of functions that satisfies both Conditions \ref{sysdef} and \ref{weaksys}.

\begin{assumption}\label{B1} 
Suppose there exist  real numbers $\sigma\ge 0$, $\alpha\ge 1/2$,  a strictly increasing, divergent sequence of positive numbers $(\gamma_n)_{n=1}^\infty$ and a sequence $(\tilde \phi_n)_{n=1}^\infty$ in $G_{\alpha,\sigma}$ such that, in the sense of Definition \ref{premsys},
\beq\label{fseq}
f(t)\simP \sum_{n=1}^\infty \tilde \phi_n \psi_{\gamma_n}(t) \text{ in }G_{\alpha,\sigma}.
\eeq
\end{assumption}

Note from \eqref{fseq} that  $f(t)$ belongs to $G_{\alpha,\sigma}$ for all $t$ sufficiently large.

Let  $u(t)$ be a Leray-Hopf weak solution of the NSE. We search for an asymptotic expansion of $u(t)$ in the form
\beq\label{try}
u(t)\simP \sum_{n=1}^\infty \xi_n \psi_{\lambda_n}(t). 
\eeq
Formally substituting expansion \eqref{try} into the NSE \eqref{fctnse}, we find that the indices $\lambda_n$'s naturally take values in the set
\beq\label{SG}
\mathcal G_\Psi(\{\gamma_n:n\in\N\}).
\eeq
However, the expansion \eqref{try} only agrees with \eqref{psexpand} in Definition \ref{premsys} if the set in \eqref{SG} does not have a finite cluster point. Therefore, we impose one more condition.

\begin{assumption}\label{C1} 
There exists a set $S_*$ that contains $\{\gamma_n:n\in\N\}$,  preserves the operations $\vee$ and $\wedge$, and can be ordered so that
\beq \label{museq}
S_*=\{\lambda_n:n\in\N\}, \text{ where $\lambda_n$'s are strictly increasing to infinity.}
\eeq 
\end{assumption}

We usually choose $S_*$ in Assumption \ref{C1} to be \eqref{SG}, but this is not the only choice.

Under Assumption \ref{C1}, we can show that the expansion \eqref{fseq} implies
\beq \label{fexp2}
f(t)\simP \sum_{n=1}^{\infty} \phi_n \psi_{\lambda_n}(t)\quad \text{in }G_{\alpha,\sigma}\quad\text{as } t\to\infty,
\eeq
where the sequence $(\phi_n)_{n=1}^\infty$ in $G_{\alpha,\sigma}$ is defined by $\phi_n=\tilde \phi_k$ if there exists $k\ge 1$ such that $\lambda_n=\gamma_k$, and $\phi_n=0$ otherwise. Note in the former case that such an index $k$, when exists, is unique.

\begin{remark}\label{sameexp}
In case the set $S=\{\gamma_n:n\in\N\}$  itself preserves the operations $\vee$ and $\wedge$, then 
$S=\mathcal G_{\Psi}(S)$. Hence, Assumption \ref{C1} is met with $S_*=S$ and \eqref{fexp2} holds with $\lambda_n=\gamma_n$, $\phi_n=\tilde\phi_n$, i.e., expansion \eqref{fexp2} is just the original \eqref{fseq}.
\end{remark}

Our first main result on the expansion of the Leray-Hopf weak solutions is the following.

\begin{theorem}\label{mainthm}
Let Assumptions \ref{B1} and \ref{C1} hold true, and let $f$ have the asymptotic expansion  \eqref{fexp2}. Then any Leray-Hopf weak solution $u(t)$  of \eqref{fctnse} 
has the asymptotic expansion
 \beq\label{uexpand}
u(t)\simP  \sum_{n=1}^\infty  \xi_n \psi_{\lambda_n}(t)\quad \text{in }G_{\alpha+1-\rho,\sigma}\text{ for all } \rho \in (0,1),
 \eeq
 where $\xi_n$'s are defined recursively by  
\begin{align} \label{xi1}
\xi_1&=A^{-1}\phi_1,\\
\label{xin1}
\xi_n&=A^{-1}\Big(\phi_n - \chi_n - \sum_{\stackrel{1\le k,m\le n-1,}{\lambda_k\wedge\lambda_m=\lambda_n}}d_{\lambda_k,\lambda_m}B(\xi_k,\xi_m)\Big) \quad\text{for } n \ge 2,
\end{align}
where 
\beq \label{chin}
\chi_n=
\begin{cases}
\begin{displaystyle}\sum_{\stackrel{(p,k)\in[1,n-1]\times \N:}{\lambda_p^\vee(k)=\lambda_n}}c_{\lambda_p,k} \xi_p\end{displaystyle},
& \text{if }\exists p\in [1, n-1],k\in\N:\lambda_p^\vee(k)=\lambda_n,\\
\qquad 0,&\text{otherwise}.
\end{cases}
\eeq 
\end{theorem}
\begin{proof}
The proof is divided into parts A, B, $\ldots$, steps 1,2, and substeps (a), (b), $\ldots$

\medskip\textbf{A.} Notation. For $n\in\N$, denote
\begin{align*}
&F_n(t)=\phi_n \psi_{\lambda_n}(t),\quad \bar F_n(t)=\sum_{j=1}^n F_j(t),\quad \text{and}\quad \tilde F_n(t)=f(t)-\bar F_n(t),\\
&u_n(t)=\xi_n \psi_{\lambda_n}(t),\quad \bar u_n(t)=\sum_{j=1}^n u_j(t),\quad\text{and}\quad v_n=u(t)-\bar u_n(t).
\end{align*}

According to the expansion \eqref{fexp2} and Definition \ref{premsys}, we can assume that  
\beq \label{Fcond}
|\tilde F_N(t)|_{\alpha,\sigma}=\mathcal O(\psi_{\lambda_N}(t)\psi_{\delta_N}(t)) ,
\eeq
for any $N\in\N$, with some $\delta_N>0$.

\medskip\textbf{B.} We observe that
 \beq\label{ximore}
 \xi_n\in  G_{\alpha+1,\sigma}\quad \forall n \ge 1.
 \eeq
The proof of \eqref{ximore} is by induction and is the same as in \cite[Lemma 4.2]{CaH1}.

By \eqref{ximore}, we have
\beq\label{ubarate}
|\bar u_n(t)|_{\alpha+1,\sigma}=\mathcal O(\psi_{\lambda_1}(t))\quad\forall n\in\N.
\eeq

\medskip\textbf{C.} As a preparation, we need to establish the large time decay for $u(t)$ first. 
Letting $N=1$ in \eqref{Fcond} gives
$|f(t)- \phi_1 \psi_{\lambda_1}(t)|_{\alpha,\sigma}=\mathcal O(\psi_{\lambda_1}(t)\psi_{\delta_1}(t) ),$
which implies
\beqs 
|f(t)|_{\alpha,\sigma}=\bigo(\psi_{\lambda_1}(t) )=\mathcal O(\psi_{\lambda_1}(t+T_*)).
\eeqs
The last relation is due to \eqref{ps2}. 

Let $F(t)=\psi_{\lambda_1}(t+T_*)$. 
Then $|f(t)|_{\alpha,\sigma}=\bigo(F(t))$, and, by \eqref{mxlam} and \eqref{ps7}, the function $F$ satisfies \eqref{eF} and \eqref{Fa}.
We now apply Theorem \ref{Fthm2} with $\varep=1/2$. 
Then  there exists  time $\hat{T}>0$ and a constant $C>0$ 
such that $u(t)$ is a regular solution of \eqref{fctnse} on $[\hat{T},\infty)$, and 
 \beq\label{ups0}
 |u(\hat{T}+t)|_{\alpha+1/2,\sigma} \le C\psi_{\lambda_1}(t+T_*)  \quad\forall t\ge 0.
 \eeq
It follows \eqref{AalphaB} and \eqref{ups0} that 
 \beq\label{Blt}
|B(u(\hat{T}+t),u(\hat{T}+t))|_{\alpha,\sigma}\le  C|u(\hat{T}+t)|_{\alpha+1/2,\sigma}^2 \le C\psi_{\lambda_1}^2(t+T_*) \quad\forall t\ge 0.
\eeq

\medskip\textbf{D.} It suffices to prove, for any $N\in\N$, that there exists a number $\varep_N>0$ such that
 \beq\label{remdelta}
|v_N(t)|_{\alpha,\sigma} =\mathcal O(\psi_{\lambda_N}(t)\psi_{\varep_N}(t) ).
 \eeq
 
We will prove \eqref{remdelta} by induction in $N$.
In calculations below, all differential equations hold in $V'$-valued distribution sense on $(T,\infty)$ for any $T>0$, which is similar to \eqref{varform}.
One can easily verify them by using \eqref{Bweak}, and the facts $u\in L^2_{\rm loc}([0,\infty),V)$ and $u'\in L^1_{\rm loc}([0,\infty),V')$ in Definition \ref{lhdef}.

\medskip
\textit{Step 1: $N=1$.} Define $w_1(t)=\psi_{\lambda_1}^{-1}(t) u(t)$. 

(a) Equation for $w_1(t)$. We have
\begin{align*}
 w_1'(t)&=\psi_{\lambda_1}^{-1}(t) u'(t) - \psi_{\lambda_1}^{-2}(t) \psi_{\lambda_1}'(t)u(t)\\
 &=\psi_{\lambda_1}^{-1}(t) \big(-Au(t)-B(u,u)+ \phi_1 \psi_{\lambda_1}(t)+\tilde F_1(t)\big) - \psi_{\lambda_1}^{-2}(t) \psi_{\lambda_1}'(t)u(t)
\end{align*}
Thus,
\beq\label{wj}
w_1'(t) + A w_1(t) = \phi_1 + H_1(t),\quad t>T_*,
\eeq
where
\beqs
H_1(t)=\psi_{\lambda_1}^{-1} [ \tilde F_1(t) - B(u(t),u(t))]- \psi_{\lambda_1}' \psi_{\lambda_1}^{-2}u(t).
\eeqs

(b) Estimation of $|H_1(t)|_{\alpha,\sigma}$. By estimates \eqref{ups0}, \eqref{Blt}, and the relations in \eqref{ps2}, we have 
\beq\label{u-1-step}
|u(t)|_{\alpha+1/2,\sigma}=\bigo(\psi_{\lambda_1}(t) ).
\eeq
\beq\label{buu-1-step}
|B(u(t),u(t)))|_{\alpha,\sigma}=\bigo(\psi^2_{\lambda_1}(t) ).
\eeq

By \eqref{Fcond}, \eqref{u-1-step}, \eqref{buu-1-step}, and properties \eqref{psi-prime}, \eqref{ps2}, there exist $T_0\ge T_*$, $\eta_1>0$, and $D_0>0$ such that for $t\ge 0$, 
\begin{align*}
\psi_{\lambda_1}^{-1}(T_0+t) |\tilde F_1(T_0+t)|_{\alpha,\sigma}
&\le D_0 \psi_{\lambda_1}^{-1}(T_0+t) \psi_{\lambda_1}(T_0+t)\psi_{\delta_1}(T_0+t)\\
&\le D_0 \psi_{\delta_1}(T_0+t),
\end{align*}
\begin{align*}
\psi_{\lambda_1}^{-1}(T_0+t)|B(u(T_0+t),u(T_0+t))|_{\alpha,\sigma}
&\le D_0 \psi_{\lambda_1}^{-1}(T_0+t)\psi_{\lambda_1}^2(T_0+t)\\
&\le D_0\psi_{\lambda_1}(T_0+t),
\end{align*}
and
\begin{align*}
\psi_{\lambda_1}'(T_0+t) \psi_{\lambda_1}^{-2}(T_0+t)|u(T_0+t)|_{\alpha+1/2,\sigma} 
&\le D_0  \psi_{\lambda_1}(T_0+t)   \psi_{\eta_1}(T_0+t) \psi_{\lambda_1}^{-2}(T_0+t) \psi_{\lambda_1}(T_0+t)\\
 & \le  D_0\psi_{\eta_1}(T_0+t).
 \end{align*}
Let $\varep_1=\min\{\delta_1,\eta_1,\lambda_1\}$. Then
\beqs
|H_1(T_0+t)|_{\alpha,\sigma} \le 3D_0 \psi_{\varep_1}(T_0+t)\quad\forall t\ge0.
\eeqs

(c) We apply Theorem \ref{Fode}(iii) to equation \eqref{wj} in $G_{\alpha,\sigma}$ with $w(t):=w_1(T_0+t)$. $f(t):=H_1(T_0+t)$, $F(t):=\psi_{\varep_1}(T_0+t)$ and $\xi=\phi_1$. We obtain from \eqref{wremain2} that
\beqs
|w_1(T_0+t)-A^{-1}\phi_1 |_{\alpha+1-\rho,\sigma}= \mathcal O (\psi_{\varep_1}(T_0+t))
\eeqs
for any $\rho\in(0,1)$, which yields
\beqs
|w_1(t)-A^{-1}\phi_1 |_{\alpha+1-\rho,\sigma}=  \mathcal O (\psi_{\varep_1}(t)).
\eeqs
Multiplying this equation by $\psi_{\lambda_1} (t)$ gives
\beqs 
|u(t)-\xi_1\psi_{\lambda_1} (t)|_{\alpha+1-\rho,\sigma}= \mathcal O (\psi_{\lambda_1}(t)\psi_{\varep_1}(t))) \quad \forall \rho \in(0,1).
\eeqs

This proves that \eqref{remdelta} holds for $N=1$.

\medskip
\textit{Step 2: Induction step.} Let $N \ge 1$ be an integer  and assume there exists $\varep_N>0$ such that
\beq\label{vNrate}
|v_N(t)|_{\alpha+1-\rho,\sigma}=\mathcal O(\psi_{\lambda_N}(t)\psi_{\varep_N}(t))\quad \forall \rho \in (0,1). 
\eeq

\medskip
(a) We will find an equation for $v_N$ which is suitable to studying its asymptotic behavior. 
First, we  have the preliminary calculations.

\medskip
\underline{\emph{Rewriting $u'$.}} By the NSE,
\begin{align*}
u'&=-Au-B(u,u)+f(t)\\
&=-Av_N - A\bar u_N -B(\bar u_N+v_N,\bar u_N+v_N)+\bar F_N+F_{N+1}  +\tilde F_{N+1}\\
&=-Av_N- A\bar u_N+\bar F_N-B(\bar u_N,\bar u_N)+ \phi_{N+1}\psi_{\lambda_{N+1}}(t) +h_{N+1,1},
\end{align*}
 where
\beqs
h_{N+1,1}=-B(\bar u_N,v_N)-B(v_N,\bar u_N)-B(v_N,v_N)+\tilde F_{N+1}.
\eeqs

On the one hand,
\beqs 
-A\bar u_N+\bar F_N= - \sum_{n=1}^N \psi_{\lambda_n}(t)\Big( A\xi_n-\phi_n\Big).
\eeqs

On the other hand,
\begin{align*}
B(\bar u_N,\bar u_N)
&=\sum_{ m,j=1}^N \psi_{\lambda_m}(t)\psi_{\lambda_j}(t)B(\xi_m,\xi_j)\\
&=\sum_{n=1}^{N}\psi_{\lambda_n}(t)\Big( \sum_{\stackrel{1\le m,j\le N,}{\lambda_m\wedge\lambda_j=\lambda_n}}d_{\lambda_m,\lambda_j}B(\xi_m,\xi_j)\Big)\\
&\quad +\psi_{\lambda_{N+1}}(t) \sum_{\stackrel{1\le m,j\le N,}{\lambda_m\wedge\lambda_j=\lambda_{N+1}}}d_{\lambda_m,\lambda_j}B(\xi_m,\xi_j)
+h_{N+1,2},
\end{align*}
where
\beqs
h_{N+1,2}=\sum_{\stackrel{1\le m,j\le N,}{\lambda_m\wedge\lambda_j\ge \lambda_{N+2}}}  \psi_{\lambda_m}(t)\psi_{\lambda_j}(t) B(\xi_m,\xi_j).
\eeqs

Then we obtain the equation
\beq\label{up}
\begin{aligned}
 u'&=-Av_N - \sum_{n=1}^N \psi_{\lambda_n}(t)\Big( A\xi_n-\phi_n+ \sum_{\stackrel{1\le m,j\le N,}{\lambda_m\wedge\lambda_j=\lambda_n}}d_{\lambda_m,\lambda_j}B(\xi_m,\xi_j)\Big)\\
&\quad -\psi_{\lambda_{N+1}}(t) \Big( \sum_{\stackrel{1\le m,j\le N,}{\lambda_m\wedge\lambda_j=\lambda_{N+1}}}d_{\lambda_m,\lambda_j}B(\xi_m,\xi_j)- \phi_{N+1}\Big) -h_{N+1,2} +h_{N+1,1},
\end{aligned}
\eeq

In calculations below to the end of this proof, $\varep$ denotes a \emph{generic} positive index used for function  $\psi_\varep(t)$.

Utilizing \eqref{Fcond}, \eqref{ubarate} and \eqref{vNrate}, we estimate, with the use of the short-hand notation $\psi_\lambda=\psi_\lambda(t)$,
\begin{align*}
 |h_{N+1,1}(t)|_{\alpha,\sigma}
 &=\mathcal O(\psi_{\lambda_1})\mathcal O(\psi_{\lambda_N}\psi_{\varep_N})
+ \mathcal O(\psi_{\lambda_N}\psi_{\varep_N}) \mathcal O(\psi_{\lambda_1})\\
&\quad + \mathcal O(\psi_{\lambda_N}\psi_{\varep_N}) \mathcal O(\psi_{\lambda_N}\psi_{\varep_N}) 
+ \mathcal O(\psi_{\lambda_{N+1}}\psi_{\delta_{N+1}})\\
&= \mathcal O(\psi_{\lambda_1}\psi_{\lambda_N}\psi_{\varep_N})  
+ \mathcal O(\psi_{\lambda_{N+1}}\psi_{\delta_{N+1}})\\
&= \mathcal O(\psi_{\lambda_1\wedge\lambda_N}\psi_{\varep_N})  
+ \mathcal O(\psi_{\lambda_{N+1}}\psi_{\delta_{N+1}}).
\end{align*}
Since $\lambda_1\wedge\lambda_N\ge \lambda_{N+1}$, we have
\beqs
|h_{N+1,1}(t)|_{\alpha,\sigma}
= \mathcal O(\psi_{\lambda_{N+1}}(t)\psi_{\varep}(t)).
\eeqs

It is also clear that 
\begin{align*}
|h_{N+1,2}(t)|_{\alpha,\sigma}
&\le\sum_{\stackrel{1\le m,j\le N,}{\lambda_m\wedge\lambda_j\ge \lambda_{N+2}}}  |d_{\lambda_m,\lambda_j}|\psi_{\lambda_m\wedge\lambda_j}(t) |B(\xi_m,\xi_j)|_{\alpha,\sigma}\\
&= \mathcal O(\psi_{\lambda_{N+2}}(t))= \mathcal O(\psi_{\lambda_{N+1}}(t)\psi_{\varepsilon}(t)).
\end{align*}

\underline{\emph{Rewriting $\bar u_N'$.}} We have
\begin{align*}
 \bar u_N'&=\sum_{p=1}^N \psi_{\lambda_p}'\xi_p = \sum_{p=1}^{N}\xi_p\Big( \sum_{k=1}^{\widetilde N_p} c_{\lambda_p,k}\psi_{\lambda_p^\vee(k)} + \psi_{\lambda_p}'-\sum_{k=1}^{\widetilde N_p} c_{\lambda_p,k}\psi_{\lambda_p^\vee(k)}\Big),
\end{align*}
where $\widetilde N_p$ is the largest integer $k$ such that
\beqs
\lambda_p^\vee(k)\le \lambda_{N+1}.
\eeqs
Then we obtain 
\beq\label{uNp}
\bar u_N'=\sum_{n=1}^{N} \psi_{\lambda_n}\chi_n+\psi_{\lambda_{N+1}}\chi_{N+1}
+h_{N+1,3},
\eeq
where
\beqs
h_{N+1,3} =\sum_{p=1}^{N}\xi_p\Big( \psi_{\lambda_p}'-\sum_{k=1}^{\widetilde N_p} c_{\lambda_p,k}\psi_{\lambda_p^\vee(k)}\Big).
\eeqs

Note from \eqref{ppex} and the definition of $ \widetilde N_p$ that 
\beqs 
|\psi_{\lambda_p}'(t)-\sum_{k=1}^{\widetilde N_p} c_{\lambda_p,k}\psi_{\lambda_p^\vee(k)}(t)|=\mathcal O(\psi_\lambda(t)),\quad \text{some } \lambda>\lambda_{N+1}.
\eeqs 
Together with \eqref{quot0}, we have
\beqs
|h_{N+1,3}(t)|_{\alpha,\sigma}
=\mathcal O(\psi_{\lambda_{N+1}}(t)\psi_{\varepsilon}(t)).
\eeqs

\medskip
\underline{\emph{Equation for $v_N$.}}  Combining \eqref{up} and \eqref{uNp} yields 
\begin{align*}
v_N'&=u'-\bar u_N'\\
&=-Av_N- \sum_{n=1}^N \psi_{\lambda_n}(t)\Big( A\xi_n+\sum_{\stackrel{1\le m,j\le N,}{\lambda_m\wedge\lambda_j=\lambda_n}}d_{\lambda_m,\lambda_j}B(\xi_m,\xi_j)-\phi_n  + \chi_n\Big) \\
&\quad +\psi_{\lambda_{N+1}}(t)\Big(- \sum_{\stackrel{1\le m,j\le N,}{\lambda_m\wedge\lambda_j=\lambda_{N+1}}}d_{\lambda_m,\lambda_j}B(\xi_m,\xi_j)+\phi_{N+1}-\chi_{N+1}\Big)
+h_{N+1,4}(t),
\end{align*}
where
\beqs
h_{N+1,4}=h_{N+1,1}-h_{N+1,2}-h_{N+1,3}. 
\eeqs

 Note, for $1\le n\le N+1$, that
 \beqs
 \sum_{\stackrel{1\le m,j\le N,}{\lambda_m\wedge\lambda_j=\lambda_n}}B(\xi_m,\xi_j)=\sum_{\stackrel{1\le  m,j \le n-1,}{\lambda_m\wedge\lambda_j=\lambda_n}}B(\xi_m,\xi_j).
 \eeqs 

Therefore,   one has, for $1\le n\le N$,
\beqs
 A\xi_n+\sum_{\stackrel{1\le m,j \le N,}{\lambda_m\wedge\lambda_j=\lambda_n}}d_{\lambda_m,\lambda_j}B(\xi_m,\xi_j)-\phi_n+\chi_n =0,
\eeqs
and 
\beqs 
-\sum_{\stackrel{1\le m,j \le N,}{\lambda_m\wedge\lambda_j=\lambda_{N+1}}}d_{\lambda_m,\lambda_j}B(\xi_m,\xi_j)+\phi_{N+1}-\chi_{N+1}=A\xi_{N+1}.
\eeqs

These yield
\beq\label{vNp}
v_N'=-Av_N + \psi_{\lambda_{N+1}}(t)A\xi_{N+1}+h_{N+1,4}(t).
\eeq

(b) Estimation of $v_N(t)$. In equation \eqref{vNp}, we have
\beq \label{h4}
|h_{N+1,4}(t)|_{\alpha,\sigma}\le |h_{N+1,1}|_{\alpha,\sigma}+|h_{N+1,2}|_{\alpha,\sigma}+|h_{N+1,3}|_{\alpha,\sigma}=\mathcal O(\psi_{\lambda_{N+1}}(t)\psi_{\varep}(t)). 
\eeq

One has from \eqref{h4} that
\beqs
|\psi_{\lambda_{N+1}}(t)A\xi_{N+1}+h_{N+1,4}(t)|_{\alpha,\sigma}
\le \psi_{\lambda_{N+1}}(t)|A\xi_{N+1}|_{\alpha,\sigma}+|h_{N+1,4}(t)|_{\alpha,\sigma}
=\mathcal O(\psi_{\lambda_{N+1}}(t)).
\eeqs

Similar to part (c) of Step 1, we apply Theorem \ref{Fode}(iii) to the linearized NSE \eqref{vNp} with $w(t)=v_N(T_0+t)$, $\xi=0$, $f(t)=\psi_{\lambda_{N+1}}(T_0+t)A\xi_{N+1}+h_{N+1,4}(T_0+t)$, and $F(t)=\psi_{\lambda_{N+1}}(T_0+t)$, where $T_0\ge T_*$ is an appropriate, sufficient large time.
We have from \eqref{wremain2}, for any $\rho\in(0,1)$, that
\beq\label{newvN}
|v_N(t)|_{\alpha+1-\rho,\sigma}=\mathcal O(\psi_{\lambda_{N+1}}(t)).
\eeq

(c) We will improve the precision of decay in \eqref{newvN}. Define $w_{N+1}(t)=\psi_{\lambda_{N+1}}(t)^{-1} v_N(t)$  for $t\ge T_*$. 
We have
\beqs 
 w_{N+1}'=\psi_{\lambda_{N+1}}^{-1}(t)v_N' +\psi_{\lambda_{N+1}}'(t)\psi_{\lambda_{N+1}}^{-2}(t) v_N,
\eeqs  
which, thanks to \eqref{vNp}, yields
\beq \label{wNeq}
w_{N+1}'=-Aw_{N+1}+A\xi_{N+1}+H_{N+1}(t),
\eeq
where
$H_{N+1}(t)
= \psi_{\lambda_{N+1}}^{-1}(t)h_{N+1,4}(t)+ \psi_{\lambda_{N+1}}'(t)\psi_{\lambda_{N+1}}^{-2}(t) v_N(t).$

\medskip
We estimate $H_{N+1}(t)$ next.
By \eqref{h4},
\beqs
|\psi_{\lambda_{N+1}}^{-1}(t)h_{N+1,4}(t)|_{\alpha,\sigma}=\psi_{\lambda_{N+1}}^{-1}(t)\mathcal O(\psi_{\lambda_{N+1}}(t)\psi_{\varep}(t))=\mathcal O(\psi_{\varep}(t)). 
\eeqs
For the second term, we use  \eqref{psi-prime} and \eqref{newvN} to obtain
\beqs
|\psi_{\lambda_{N+1}}'(t)\psi_{\lambda_{N+1}}^{-2}(t) v_N(t)|_{\alpha,\sigma}
=\psi_{\lambda_{N+1}}^{-2}(t)\mathcal O (\psi_{ \lambda_{N+1}}(t)\psi_\varep(t))
\mathcal O(\psi_{\lambda_{N+1}}(t))
=\mathcal O(\psi_\varep(t)).
\eeqs
Hence, there exists $\varep_{N+1}>0$ such that
\beqs 
|H_{N+1}(t)|_{\alpha ,\sigma}
=\mathcal O(\psi_{\varep_{N+1}}(t)).
\eeqs 

(d) Note  from \eqref{ximore} that $A\xi_{N+1}\in G_{\alpha,\sigma} \subset G_{\alpha-\frac 1 2,\sigma} $. 
Again, by applying Theorem \ref{Fode}(iii) to equation \eqref{wNeq} with $w(t):=w_{N+1}(T_1+t)$, $\xi:=A\xi_{N+1}$, $f(t):=H_{N+1}(t+T_1)$, $F(t):=\psi_{\varep_{N+1}}(t+T_1)$ for some  $T_1\ge T_*$ sufficiently large, we obtain from \eqref{wremain2}, for any $\rho\in(0,1)$, that
\beqs
|w_{N+1}(T_1+t)-A^{-1}(A\xi_{N+1})|_{\alpha+1-\rho,\sigma} \le  C \psi_{\varep_{N+1}}(t+T_1)\quad\forall t\ge 1.
\eeqs
Thus, 
$|w_{N+1}(t)-\xi_{N+1}|_{\alpha+1-\rho,\sigma}=   \mathcal O(\psi_{\varep_{N+1}}(t)).$
Multiplying this equation by $\psi_{\lambda_{N+1}}(t)$ yields
\beqs
|v_N(t)- \xi_{N+1}\psi_{\lambda_{N+1}}(t)|_{\alpha+1 -\rho,\sigma}=\mathcal O(\psi_{\lambda_{N+1}}(t)\psi_{\varepsilon_{N+1}}(t)).
\eeqs

Since the left-hand side of this equation is $|v_{N+1}(t)|_{\alpha+1 -\rho,\sigma}$, 
it proves that the statement \eqref{remdelta} holds true for $N:=N+1$.

\medskip
\textit{Conclusion.} By the induction principle, we have \eqref{remdelta} holds true for all $N\in\N$.
Our proof is complete.
\end{proof}

 In Theorem \ref{mainthm}, both force $f$ and solution $u$ have  infinite sum expansions which means 
that they can be approximated by infinitely many terms $\psi_\lambda$'s as $\lambda\to\infty$. 
The case of  finite sum approximations can be treated similarly. 
We briefly discuss the idea and result here.

\begin{assumption}\label{D1}
Suppose there exist numbers $\sigma\ge 0$, $\alpha\ge 1/2$,  an integer $N_0\ge1$, strictly increasing, positive  numbers $\gamma_n$ and functions  $\tilde\phi_n\in G_{\alpha,\sigma}$ for $1\le n\le N_0$  such that 
\beq \label{fin2} 
\Big|f(t) - \sum_{n=1}^{N_0} \tilde \phi_n \psi_{\gamma_n}(t)\Big|_{\alpha,\sigma} =\mathcal O(\psi_\lambda(t))\text{ for some } \lambda>\gamma_{N_0}.
\eeq

Assume further that there exists a set $S_\infty$ that contains $\{\gamma_n:1\le n\le N_0\}$ and preserves the operations $\vee$ and $\wedge$, so that the set $S_*\eqdef S_\infty \cap [\gamma_1,\gamma_{N_0}]$ is finite.
\end{assumption}

In applications, we often choose $S_\infty=\mathcal G_{\Psi}(\{\gamma_n:1\le n\le N_0\})$, 
but it can be more general than this.

We rewrite $ S_*=\{\lambda_n: 1 \le n\le N_*\}$ for some integer $N_*\ge N_0$, where $\lambda_n$'s are strictly increasing. Note that $\lambda_{N_*}=\gamma_{N_0}$.
Then from \eqref{fin2} we have
\beq \label{f4}
\Big|f(t) - \sum_{n=1}^{N_*} \phi_n \psi_{\lambda_n}(t)\Big|_{\alpha,\sigma} =\mathcal O(\psi_\lambda(t))\text{ for some } \lambda>\lambda_{N_*}.
\eeq
where $\phi_n\in G_{\alpha,\sigma}$ for all $1\le n\le N_*$.

\begin{theorem} \label{finthm1}
Let Assumption \ref{D1} hold true, and let $f$ have the asymptotic approximation \eqref{f4}.
Let $\xi_n$ be defined by \eqref{xi1} and \eqref{xin1} for $1\le n\le N_*$.
For any Leray-Hopf weak solution $u(t)$ of \eqref{fctnse}, it holds that
$$\Big|u(t) - \sum_{n=1}^{N_*} \xi_n \psi_{\lambda_n}(t)\Big|_{\alpha,\sigma} =\mathcal O(\psi_\lambda(t))\text{ for some } \lambda>\lambda_{N_*}.
$$
\end{theorem}
\begin{proof}
The proof of Theorem \ref{finthm1} is the same as that of Theorem \ref{mainthm} except that we only use \emph{finite} induction to establish \eqref{remdelta} for $1\le N\le N_*$.
\end{proof}

\section{Asymptotic expansions in a discrete system with a continuum background}\label{discrtsec}

In this section, we investigate the case that the system of functions $(\psi_n)_{n=1}^\infty$ in Definition \ref{sys0} cannot be  mapped directly to a system $(\psi_{\lambda_n})_{n=1}^\infty$ to be embedded into a continuum system $(\psi_\lambda)_{\lambda>0}$. Hence, Definition \ref{premsys} will not apply. However, we consider below the case when 
each $\psi_n$ is of the same decaying order, when $t\to\infty$, as $\varphi_{\lambda_n}$ with the functions $\varphi_{\lambda_n}$'s being part of a continuum system $(\varphi_\lambda)_{\lambda>0}$.

\begin{definition}\label{dpremsys}
Let $\Psi=(\psi_n)_{n=1}^\infty$ be a sequence of positive functions defined on $[T_*,\infty)$ for some $T_*\in\R$,
and $\Phi=(\varphi_{\lambda})_{\lambda>0}$ be a continuum system as in Definition \ref{premsys} such that 
there exists a strictly increasing, divergent sequence $(\lambda_n)_{n=1}^\infty$ of positive numbers such that
\beq\label{psph}
\psi_n(t)\Oeq \varphi_{\lambda_n}(t)\quad\text{for all }n\in\N.
\eeq

Let $(X, \|\cdot\|)$ be a normed space, and $g$ be a function from $(0,\infty)$ to $X$.
We define the asymptotic expansions
\beq\label{Phiex} 
g(t) \dsim{\Phi} \sum_{n=1}^\infty \xi_n \psi_n(t),\quad 
g(t) \dsim{\Phi} \sum_{n=1}^N \xi_n \psi_n(t) \text{ with }N\in\N,\quad 
g(t) \dsim{\Phi} 0
\text{ in } X,
\eeq
in the same way as Definition \ref{premsys} and the special cases {\rm(i)--(iv)} below it,
where we replace $\psi_{\lambda_n}$ with $\psi_n$, 
replace \eqref{lamrate} with 
\beq \label{Prate}
\Big\|g(t)- \sum_{n=1}^N \xi_n \psi_n(t)\Big\|=\mathcal O(\psi_N(t)\varphi_\varep(t)),
\eeq
replace \eqref{finiterate} with 
\beq \label{Pfin}
\Big\|g(t)- \sum_{n=1}^N \xi_n \psi_n(t)\Big\|=\mathcal O(\varphi_\lambda(t)),
\eeq
replace  \eqref{zeroexp} with 
\beq \label{Pzero}
\|g(t)\|=\mathcal O(\varphi_\lambda(t)).
\eeq
\end{definition}

We refer to $\Phi$ as a \emph{background system} of $(\psi_n)_{n=1}^\infty$.
We will write  expansions in \eqref{Phiex} as
\beqs
g(t) \dsim{\Phi} \sum_{n=1}^N \xi_n \psi_n(t) \text{ in } X,\text{ for $N=\infty$, $N\in\N$, and $N=0$, respectively.}
\eeqs

The following remarks on Definition \ref{dpremsys} are in order.

\begin{enumerate}[label={\rm(\alph*)}]
 \item It follows \eqref{psph} immediately that property \eqref{limzero} holds true.
Moreover, for any $n>m$, there exists $\eta>0$ such that
 \beq \label{dquot}
 \psi_n(t)=\mathcal O(\psi_m(t) \varphi_\eta(t)).
 \eeq
 where, $\eta$ is a number such that
$\phi_{\lambda_n}(t)=\mathcal O(\phi_{\lambda_m}(t)\psi_\eta(t))$, thanks to property \eqref{quot0} for the system $\Phi$.
Thus, property \eqref{noquot} is also true.
Therefore, Definition \ref{sys0} for the asymptotic expansions $g(t)\sim \sum_{n=1}^\infty \xi_n \psi_n(t)$  in $X$ still applies.

Obviously, if $g(t) \dsim{\Phi} \sum_{n=1}^N \xi_n \psi_n(t)$ then $g(t)\sim \sum_{n=1}^N \xi_n \psi_n(t)$ in the sense of Definition \ref{sys0}.
Then, thanks to Proposition \ref{Prop62}, the uniqueness of the latter expansion implies the uniqueness of the former one.

\item We can \emph{equivalently} replace $\mathcal O(\psi_N(t)\varphi_\varep(t))$ in \eqref{Prate} with
$\mathcal O(\varphi_{\lambda_N}(t)\varphi_\varep(t))$, or  $\mathcal O(\varphi_\lambda(t))$ for some $\lambda>\lambda_N$.

\item For a given sequence  $(\psi_n)_{n=1}^\infty$, there may be different background systems. However, the asymptotic expansion of a function $g$ as defined in \eqref{Phiex}, thanks to remark (a), is unique disregarding the choice of the background system $\Phi$.

\item Let $\Phi=(\varphi_\lambda)_{\lambda>0}$ and $\Theta=(\vartheta_\lambda)_{\lambda>0}$ be two systems  as in Definition \ref{premsys}. If there exists a strictly increasing bijection $\mu$ from $(0,\infty)$ to $(0,\infty)$
such that $\varphi_\lambda(t)\Oeq \vartheta_{\mu(\lambda)}(t)$ for all $\lambda>0$, then 
$$ g(t) \dsim{\Phi} \sum_{n=1}^N \xi_n \psi_n(t) \quad \text{ if and only if }\quad g(t) \dsim{\Theta} \sum_{n=1}^N \xi_n \psi_n(t) .  $$

\item 
Let $\Psi=(\psi_\lambda)_{\lambda>0}$ and $\Phi=(\varphi_\lambda)_{\lambda>0}$ satisfy (a) and (b) of Definition \ref{premsys}.
Suppose there exists a strictly increasing bijection $\mu$ from $(0,\infty)$ to $(0,\infty)$ such that 
\beq\label{ppmu}
\psi_\lambda(t)\Oeq \varphi_{\mu(\lambda)}(t)\text{ for all }\lambda>0.
\eeq
Let $X$, $(\xi_n)_{n=1}^\infty$ and $(\lambda_n)_{n=1}^\infty$ be as in Definition \ref{premsys}. Set $\tilde \psi_n=\psi_{\lambda_n}$ for all $n\in\N$. 
If $g$ is a function from $(0,\infty)$ to $X$, then
\beq\label{equiex1}
g(t)\simP \sum_{n=1}^\infty \xi_n\psi_{\lambda_n}(t)\text{ if and only if }
g(t)\dsim{\Phi} \sum_{n=1}^\infty \xi_n\tilde \psi_n(t).
\eeq
For simplicity, we will write the last expansion as
\beq\label{equiex2}
g(t)\dsim{\Phi} \sum_{n=1}^\infty \xi_n\psi_{\lambda_n}(t).
\eeq
\end{enumerate}

While the functions $\psi_n$'s, with discrete index $n$, are the actual functions presented in the expansions in \eqref{Phiex}, the functions $\varphi_\lambda$, with the continuum index $\lambda$, provides specific rates in comparison \eqref{dquot} and remainder estimates \eqref{Pzero}, \eqref{Pfin}, \eqref{Prate}. The fact that $\lambda$ has the range $(0,\infty)$ gives $\varphi_\lambda$ the flexibility in many comparisons and estimates, while the structure of the expansions is maintained by $\psi_n$'s. 
Note also that $\psi_n$'s are not required to be decreasing anymore.

\begin{assumption}\label{dsysdef} 
For the rest of this section, we assume that $\Psi=(\psi_n)_{n=1}^\infty$ and $\Phi=(\varphi_\lambda)_{\lambda>0}$  are a pair of systems as in Definition \ref{dpremsys} that further satisfy
\begin{enumerate}[label=\rm (\roman*)]
 \item For any $m,n\in\N$, there exist a natural number $k > \max\{m,n\}$ and a nonzero constant $d_{m,n}$ such that 
 \beq\label{dtriple} \psi_m\psi_n=d_{m,n}\psi_k.
 \eeq 

 \item For each $n\in\N$, $\psi_n$ is continuous and differentiable on $[T_*,\infty)$, and  $\psi_n'$ has an expansion in the sense of Definition \ref{dpremsys}
 \beqs
 \psi_n'(t)\dsim{\Phi} \sum_{k=n+1}^{N_n} c_{n,k}\psi_{k}(t) \text{ in }\R,
 \eeqs
  where $N_n\in\N\cup\{0,\infty\}$, all $c_{n,k}$ are constants.
  
  \item The system $\Phi=(\varphi_\lambda)_{\lambda>0}$ satisfies Condition \ref{weaksys}.
\end{enumerate}
\end{assumption}

\begin{notation}
 We denote the unique number $k$ in \eqref{dtriple} by $m\wedge n$.
\end{notation}

We obtain the asymptotic expansions of type \eqref{Phiex} for the NSE.

\begin{theorem}\label{discretethm}
Suppose there exist $\alpha\ge 1/2$, $\sigma\ge 0$, and $\phi_n\in G_{\alpha,\sigma}$ for all $n\in\N$ such that
\beqs
f(t)\dsim{\Phi} \sum_{n=1}^\infty \phi_n \psi_n(t)\text{ in  }G_{\alpha,\sigma}.
\eeqs

Then any Leray-Hopf weak solution $u(t)$  of \eqref{fctnse} 
has the asymptotic expansion
 \beq\label{uexp2}
u(t)\dsim{\Phi} \sum_{n=1}^\infty  \xi_n \psi_n(t)\quad \text{in }G_{\alpha+1-\rho,\sigma} \text{ for all } \rho \in (0,1),
 \eeq
 where 
\beq\label{dxin}
 \xi_1=A^{-1}\phi_1,\quad 
 \xi_n=A^{-1}\Big(\phi_n - \chi_n -  \begin{displaystyle}\sum_{\stackrel{1\le k,m\le n-1,}{k\wedge m=n}} \end{displaystyle}
 d_{k,m}B(\xi_k,\xi_m)\Big)
\text{ for }n \ge 2, 
\eeq
with $\chi_n=\sum_{p=1}^{n-1}c_{p,n} \xi_p$.
\end{theorem}
\begin{proof}

We follow the proof of   Theorem \ref{mainthm} and make the following replacements:  
\begin{itemize}
 \item  $\psi_{\lambda_n}$ is replaced with $\psi_n$ for all $n\in\N$, and
 \item $\psi_\sharp$ is replaced with $\varphi_\sharp$ whenever the subscript symbol $\sharp$ is $\delta_1$, $\varep$, $\varep_1$, $\eta_1$, $\varep_{N}$, $\varep_{N+1}$, $\delta_{N+1}$.
 \end{itemize}
 
 It results in the expansion \eqref{uexp2} as desired.
 \end{proof}

For finite sum asymptotic approximations in a discrete system, we obtain the following counter part of Theorem \ref{finthm1}.

\begin{theorem} \label{finthm2}
Suppose there exist numbers $\sigma\ge 0$, $\alpha\ge 1/2$, $N_*\in \N$, and functions  $\phi_n\in G_{\alpha,\sigma}$ for $1\le n\le N_*$  such that 
\beqs \label{fin3} 
\Big|f(t) - \sum_{n=1}^{N_*} \phi_n \psi_n(t)\Big|_{\alpha,\sigma} =\mathcal O(\varphi_\lambda(t))\text{ for some } \lambda>\lambda_{N_*}.
\eeqs 

Let $\xi_n$ be defined by \eqref{dxin} for $1\le n\le N_*$.
Then any Leray-Hopf weak solution $u(t)$ satisfies
$$\Big|u(t) - \sum_{n=1}^{N_*} \xi_n \psi_n(t)\Big|_{\alpha,\sigma} =\mathcal O(\varphi_\lambda(t))\text{ for some } \lambda>\lambda_{N_*}.
$$
\end{theorem}
\begin{proof}
The proof of Theorem \ref{finthm2} is the same as that of Theorem \ref{finthm1} with the use of replacements in the proof of Theorem \ref{discretethm}.
\end{proof}

\section{Applications}\label{examsec}

We will apply results in Sections \ref{expsec} and \ref{discrtsec} to obtain specific expansions for solutions of the NSE corresponding to different types of forces.
We focus on the infinite expansions, hence, show only applications of Theorems \ref{mainthm} and \ref{discretethm}. Their counterparts using the finite asymptotic approximations in Theorems \ref{finthm1} and \ref{finthm2} can be similarly obtained. However, they will not be presented here, for the sake of avoiding repetitions and keeping the paper concise.

\medskip
First, we discuss a very frequently used type of systems of functions for long-time asymptotic expansions.

\begin{definition}
 A \emph{P-system} is a system $\Psi=(\psi_\lambda)_{\lambda>0}$, with $\psi_\lambda=\varphi^\lambda$, where 
$\varphi$ is a positive function defined on $[T_*,\infty)$ for some $T_*\ge 0$, and $\varphi(t)\to 0$ as $t\to\infty$.
\end{definition}

\noindent\textbf{Property (P).} Clearly, a P-system $\Psi$ satisfies (a) and (b) in Definition \ref{premsys} with $\eta=\lambda-\mu$, and (i) in Condition \ref{sysdef} with $d_{\lambda,\mu}=1$ and $\gamma=\lambda\wedge \mu=\lambda+\mu$.

In this case, a set $S\subset (0,\infty)$ preserves the operation $\wedge$, see \eqref{S2}, if and only if it \emph{preserves the addition}, i.e.,  $x+y\in S$ whenever $x,y\in S$.

\medskip
In subsections \ref{ss1}--\ref{ss4} below, let $\sigma\ge 0$ and $\alpha\ge 1/2$ be given numbers, $(\gamma_n)_{n=1}^\infty$ be a strictly increasing, divergent sequence of positive numbers, and $(\tilde \phi_n)_{n=1}^\infty$ be  a sequence in $G_{\alpha,\sigma}$.

\subsection{The system of power-decaying functions}\label{ss1}
We quickly demonstrate how to apply Theorem \ref{mainthm} to recover one of the main theorems in \cite{CaH1} on the expansions in the system of power-decaying functions.

Let $\Psi=(t^{-\lambda})_{\lambda>0}$ which  is a P-system.
\begin{enumerate}[label={\rm(\roman*)}]
 \item By Property (P), $\Psi$ satisfies (a) and (b) of Definition \ref{premsys}.
 \item By Property (P), $\Psi$ satisfies (i) of Condition \ref{sysdef}.
 In addition, it satisfies (ii) of Condition \ref{sysdef} with  
\beq\label{Pvee}
N_\lambda=1, \quad c_{\lambda,1}=-\lambda,\quad \lambda^\vee(1)=\lambda+1\text{ for all }\lambda>0.
\eeq
Thus, $\Psi$ meets Condition \ref{sysdef}. 
 \item Elementary calculations show $\Psi$ meets Condition \ref{weaksys}.
\end{enumerate}

Therefore, $\Psi$ satisfies the conditions set from the beginning of Section \ref{expsec}.

\medskip
Note from \eqref{Pvee} that a set $S\subset (0,\infty)$ preserves the operation $\vee$, see \eqref{S1}, if and only if it \emph{preserves the  increments by $1$}, i.e., $x+1\in S$ whenever $x\in S$.

\medskip
We assume the force has an expansion
\beq\label{f-power}
f(t) \simP \sum_{n=1}^\infty \tilde \phi_n t^{-\gamma_n} \quad \text{ in }G_{\alpha,\sigma}.
\eeq

Let 
\beq\label{Sst1}
S_*=\Big \{  \sum_{j=1}^p \gamma_{n_j} +k : \ p,n_1,n_2,\ldots,n_p\in\N,\ k\in\N\cup\{0\}\Big\}. 
\eeq

Clearly, the set  $S_*$ in \eqref{Sst1}  satisfies Assumption \ref{C1}. We assume \eqref{museq}, and rewrite \eqref{f-power} as
\beqs \label{f-power2}
f(t)\simP \sum_{n=1}^{\infty} \phi_n  t^{-\lambda_n}\quad \text{in }G_{\alpha,\sigma},
\eeqs
where the sequence $(\phi_n)_{n=1}^\infty\subset G_{\alpha,\sigma}$ is defined as in \eqref{fexp2}.
Then Theorem \ref{mainthm} implies that any Leray-Hopf weak solution $u(t)$ of the NSE \eqref{fctnse} has the asymptotic expansion 
\beqs
u(t) \simP \sum_{n=1}^\infty  \xi_n t^{-\lambda_n}\quad \text{in }G_{\alpha+1-\rho,\sigma}\text{ for all } \rho \in (0,1),
 \eeqs
where  
\beqs
\xi_1=A^{-1}\phi_1,\quad \xi_n=A^{-1}\Big(\phi_n + \chi_n - \sum_{\stackrel{1\le k,m\le n-1,}{\lambda_k+\lambda_m=\lambda_n}}B(\xi_k,\xi_m)\Big) \quad\text{for } n\ge 2,
\eeqs
with $\chi_n=\lambda_p \xi_p$ if there exists an integer  $p\in [1, n-1]$ such that $\lambda_p +1= \lambda_n$,
and $\chi_n=0$ otherwise.


We have recovered Theorem 4.3 in \cite{CaH1} as a consequence of Theorem \ref{mainthm}.

\subsection{Systems of iterated logarithmic, decaying functions}\label{ss2}

 We consider the case when the force decays as logarithmic or iterated logarithmic functions. 
 
For $k,m \in \N$, let 
\beqs
L_k(t)= \underbrace{\ln(\ln(\cdots \ln(t)))}_{\text{$k$-times}}\quad \text{and}\quad 
\mathcal{L}_m(t)= (L_1(t),L_2(t), \cdots, L_m(t)).
\eeqs

Let $Q_0: \R^m \to R$ be a polynomial in $m$ variables:  
\beq\label{P0}
Q_0(z)=\sum_\alpha c_\alpha z^\alpha \text{ for }z\in\R^m,
\eeq
where the sum is taken over finitely many multi-index $\alpha=(\alpha_1,\alpha_2,\ldots,\alpha_m)$, and $c_\alpha$'s are (real) constants.
We use the lexicographic order for the multi-indices in \eqref{P0}. 

We assume that $Q_0(z)$ has positive degree and positive leading coefficient.
Denote by $\alpha_*=(\alpha_{*1},\alpha_{*2},\ldots,\alpha_{*m})$ the largest multi-index (with the lexicographic order) in \eqref{P0} for which $c_{\alpha_*}\ne 0$.
Then we have $|\alpha_*|\ge 1$ and $c_{\alpha_*}>0$.

Let $Q_1$ be a polynomial in one variable of positive degree with positive leading coefficient.
Denote the degree of $Q_1$ by $d\ge 1$, and the leading coefficient by $a_d>0$.

Given a number $\beta>0$, we define
\beq\label{philog}
\omega(t)= (Q_0 \circ \mathcal{L}_m \circ Q_1)(t^\beta))\text{ with } t\in \R.
\eeq

One can see that there exists $T_*>0$ such that $\omega$ is  
a positive function defined on $[T_*,\infty)$ and $\omega(t)\to \infty$ as $t\to\infty$.

Let $\psi_\lambda(t)= \omega(t)^{-\lambda}$ for $\lambda>0$, and let $\Psi$ be the P-system $(\psi_\lambda)_{\lambda>0}$.

\begin{lemma}\label{lemL}
If $\lambda>0$ then
\beq\label{psilog}
\lim_{t\to\infty} \psi_\lambda(t) \mathcal L_m( t)^{\lambda\alpha_*}=(c_{\alpha_*}(\beta d)^{\alpha_{*1}})^{-\lambda}.
\eeq
\end{lemma}
\begin{proof}
 First, if $k<j$ then $L_j(t)=o(L_k(t))$. 
With the lexicographic order, we have
$$\lim_{t\to\infty}\frac{ (Q_0\circ \mathcal L_m)(t)}{ c_{\alpha_*} \mathcal L_m(t)^{\alpha_*}}=1,
\text{ which implies }
\lim_{t\to\infty} \frac{ \omega(t)}{ \mathcal L_m(Q_1(t^\beta))^{\alpha_*}}=c_{\alpha_*}>0.$$
Moreover,
$$\lim_{t\to\infty} \frac{ \mathcal L_m(Q_1(t^\beta))^{\alpha_*}}{ \mathcal L_m(a_d t^{\beta d})^{\alpha_*}}=1.$$
By the properties of the logarithmic function, one has, for any $a,r>0$, that 
\beq\label{Lka}
\lim_{t\to\infty} \frac{L_k(at^r)}{L_k(t)}=\begin{cases}
                                            r,&\text{ for } k=1,\\
                                            1,&\text{for } k>1.
                                           \end{cases}
\eeq
Combining these gives
\beqs
\lim_{t\to\infty}\frac{ \omega(t)}{\mathcal L_m( t)^{\alpha_*}}
=\lim_{t\to\infty} \frac{ \omega(t)}{ \mathcal L_m(Q_1(t^\beta))^{\alpha_*}}\cdot 
\lim_{t\to\infty} \frac{ \mathcal L_m(Q_1(t^\beta))^{\alpha_*}}{ \mathcal L_m(a_d t^{\beta d})^{\alpha_*}}\cdot
\lim_{t\to\infty} \frac{\mathcal L_m(a_d t^{\beta d})^{\alpha_*}}{\mathcal L_m(t)^{\alpha_*}}
=c_{\alpha_*} (\beta d)^{\alpha_{*1}}.
\eeqs
Thus, \eqref{psilog} follows.
\end{proof}

As a consequence of \eqref{psilog}, we have
\beq\label{psiO}
 \psi_\lambda(t)\Oeq \mathcal L_m( t)^{-\lambda\alpha_*}.
 \eeq

 In particular, if $\alpha_*= p_0 e_k$ for some $p_0\in\N$, where $e_k$ is the $k$-th unit vector of the canonical basis of $\R^m$, then 
\beq\label{justLk}  
\psi_\lambda(t)\Oeq L_k(t)^{-p_0 \lambda}.
\eeq

We verify Conditions \ref{sysdef} and  \ref{weaksys} for the P-system $\Psi$.

\medskip
\noindent\textit{Verification of Condition \ref{sysdef}.}
Because of Property (P) for $\Psi$, we only need to check (ii) of Condition \ref{sysdef}.
By Chain Rule, 
\begin{align*}
 \ddt (\mathcal{L}_m(t)^\alpha)
&=\sum_{k=1}^m \alpha_k \mathcal{L}_m(t)^{\alpha-e_k}\frac{d}{dt}L_k(t)
=\sum_{k=1}^m \alpha_k \mathcal{L}_m(t)^{\alpha-e_k}\frac{1}{t \Pi_{p=1}^{k-1}L_p(t)}\\
&=\frac1t\sum_{k=1}^m \alpha_k \mathcal{L}_m(t)^{\alpha-e_1-e_2-\ldots e_k}.
\end{align*}
Then,
\begin{align*}
\frac{d}{dt}(\mathcal L_m(Q_1(t^\beta))^\alpha)&= \frac{\beta t^{\beta-1} Q_1'(t^\beta)}{Q_1(t^\beta)}\sum_{k=1}^m \alpha_k \mathcal{L}_m(Q_1(t^\beta))^{\alpha-e_1-e_2-\ldots -e_k}.
\end{align*}
We estimate
\beqs
|\frac{d}{dt}(\mathcal L_m(Q_1(t^\beta))^\alpha)|
=\mathcal O(t^{-1}\mathcal L_m(Q_1(t^\beta))^{\alpha_*})
=\mathcal O(t^{-1}\mathcal L_m(t))^{\alpha_*}).
\eeqs
Hence,
\beqs
|\omega'(t)|
=\mathcal O(t^{-1}\mathcal L_m(t))^{\alpha_*}).
\eeqs
Since $\psi_\lambda'(t)=-\lambda \omega(t)^{-\lambda-1} \omega'(t)$, we combine these with Lemma \ref{lemL} to obtain
\beqs  
|\psi_\lambda'(t)| = \mathcal{O} \left(\frac{\mathcal{L}_m(t)^{\alpha_*}}{t \cdot \mathcal{L}_m(t)^{(1+\lambda)\alpha_*}}\right)
= \mathcal{O} (t^{-1}).
\eeqs
This implies
\beqs 
|\psi_\lambda'(t)|
= \mathcal{O} \left(\mathcal{L}_m(t)^{-\mu \alpha_*}\right)=\mathcal O(\psi_\mu(t)), \quad \forall \mu >0.
\eeqs 
Therefore, by definition,
\beq \label{pszero}
\psi_\lambda'(t) \simP 0\text{ for all }\lambda>0.
\eeq 
Thus, $\Psi$ satisfies (ii) of Condition \ref{sysdef}
with $N_\lambda=0$ for all $\lambda>0$.


\medskip
\noindent\textit{Verification of Condition \ref{weaksys}.}
Thanks to \eqref{psilog} and \eqref{Lka}, the requirements (ii) and (iii) are met. 
For (i), using the above calculations we find
\begin{align*}
\omega'(t)&= \frac{\beta t^{\beta-1} Q_1'(t^\beta)}{Q_1(t^\beta)}\sum_{\alpha=(\alpha_1,\alpha_2,\ldots,\alpha_m)} c_\alpha \sum_{k=1}^m \alpha_k \mathcal{L}_m(Q_1(t^\beta))^{\alpha-e_1-e_2-\ldots e_k}.
\end{align*}

Let $\gamma_*$ be the largest multi-index among $\alpha-e_1-e_2-\ldots -e_k$ with non-zero $c_\alpha\alpha_k$.

Then $\gamma_*=\alpha_*-e_1-e_2-\ldots -e_k$ where $k$ is the smallest index for which the component $\alpha_{*k}\ge 1$.
Hence the corresponding coefficient is $c_{\alpha_*}\alpha_{*k}>0$.
Note also that $Q_1'(t^\beta)>0$ for large $t$. 
We conclude, for sufficiently large $t$, that $\omega'(t)>0$, and hence $\psi_\lambda'(t)<0$ .

\medskip
Now, we assume the force has the following expansion
\beqs
f(t) \simP  \sum_{n=1}^\infty \tilde \phi_n \omega(t)^{-\gamma_n}= \sum_{n=1}^\infty \tilde \phi_n \psi_{\gamma_n}(t) \text{ in }G_{\alpha,\sigma}.
\eeqs 

Let $S_*= \{  \sum_{j=1}^p \gamma_{n_j} :  p,n_1,n_2,\ldots,n_p\in\N\}$.
Then, again, this set $S_*$ satisfies Assumption \ref{C1}, hence we can assume \eqref{museq} and the expansion \eqref{fexp2}, which reads as
\beq \label{f-log2}
f(t)\simP \sum_{n=1}^{\infty} \phi_n \omega(t)^{-\lambda_n}\quad \text{in }G_{\alpha,\sigma},
\eeq
for a sequence $(\phi_n)_{n=1}^\infty$ in $G_{\alpha,\sigma}$.


\begin{theorem}\label{logcase}
Let $\varphi$ be defined by \eqref{philog} and assume the expansion \eqref{f-log2}. 
\begin{enumerate}[label={\rm(\roman*)}]
 \item 
 Then any Leray-Hopf weak solution $u(t)$ of the NSE \eqref{fctnse} has the asymptotic expansion
\beq\label{u-log}
u(t)\simP  \sum_{n=1}^\infty  \xi_n  \omega(t)^{-\lambda_n}\quad \text{in }G_{\alpha+1-\rho,\sigma}\text{ for all } \rho \in (0,1),
 \eeq
where 
\beq\label{xin-log}
\xi_1=A^{-1}\phi_1,\quad 
\xi_n=A^{-1}\Big(\phi_n - \sum_{\stackrel{1\le k,m\le n-1,}{\lambda_k+\lambda_m=\lambda_n}}B(\xi_k,\xi_m)\Big) \quad\text{for } n\ge 2.
\eeq

\item  
By defining $\varphi_\lambda(t)=\mathcal L_m(t)^{-\lambda\alpha_*}$ and $\Phi=(\varphi_\lambda)_{\lambda>0}$, we can \emph{equivalently} replace $\simP$ with $\dsim{\Phi}$ in \eqref{f-log2} and \eqref{u-log}, in the sense of  \eqref{equiex2}.

In particular, if $\alpha_*$ is co-linear with the $k$-th unit vector $e_k$ of the canonical basis of $\R^m$, 
then this replacement still holds true for 
$\Phi=(L_k(t)^{-\lambda})_{\lambda>0}$.
\end{enumerate}
\end{theorem}
\begin{proof} (i) Applying Theorem \ref{mainthm} while noting that $\chi_n =0$ in \eqref{chin} for all $n \in \N$ due to the fact  \eqref{pszero}, we deduce \eqref{u-log} from \eqref{uexpand}.
 
(ii) Let $\varphi_\lambda(t)=\mathcal L_m(t)^{-\lambda\alpha_*}$ and $\Phi=(\varphi_\lambda)_{\lambda>0}$. Note by \eqref{psiO} that
$\psi_\lambda(t)\Oeq \varphi_\lambda(t)$ which implies \eqref{ppmu} with $\mu(\lambda)=\lambda$.
Then the replacement is valid thanks to \eqref{equiex1} and \eqref{equiex2} in remark (e) after Definition \ref{dpremsys}.

Now, consider the case when $\alpha_*=p_0 e_k$ for some $p_0\in \N$. Let $\varphi_{\lambda}(t)=L_k(t)^{-\lambda}$ and, again,  $\Phi=(\varphi_\lambda)_{\lambda>0}$.
By \eqref{justLk}, we have
\beqs
\psi_\lambda(t)\Oeq \varphi_{p_0\lambda}(t)=\varphi_{\mu(\lambda)}(t),\quad\text{where }\mu(\lambda)=p_0\lambda.
\eeqs
Hence, the replacement is valid again by the same  \eqref{equiex1} and \eqref{equiex2}.
%
%
\end{proof}

\begin{corollary}\label{purelog}
Given $m\in \N$, define $\Psi=(L_m(t)^{-\lambda})_{\lambda>0}$.  
Suppose $(\lambda_n)_{n=1}^\infty$ is a strictly increasing, divergent sequence of positive numbers such that the set  $\{\lambda_n:n\in\N\}$ preserves the addition.
If 
\beqs 
f(t)\simP \sum_{n=1}^{\infty} \phi_n L_m(t)^{-\lambda_n}\quad \text{in }G_{\alpha,\sigma},
\eeqs
 then any Leray-Hopf weak solution $u(t)$ of the NSE \eqref{fctnse} admits the same asymptotic expansion
\beqs 
u(t)\simP \sum_{n=1}^{\infty} \xi_n L_m(t)^{-\lambda_n}\quad \text{in }G_{\alpha+1-\rho,\sigma}\text{ for all } \rho \in (0,1),
\eeqs
where $\xi_n$'s are defined by \eqref{xin-log}.
\end{corollary}
\begin{proof}
We choose $Q_0(z_1,z_2,\cdots,z_m)=z_m$, $Q_1(t)=t$ and $\beta=1$.
Notice, in this case,  that \eqref{philog} becomes $\omega(t)=L_m(t)$. 
Then the result in this corollary follows Theorem \ref{logcase}(i) and Remark \ref{sameexp}.
\end{proof}


\begin{example} \label{log-eg}
Thanks to Theorem \ref{logcase}, we can have asymptotic expansions of many different types. We illustrate it with just two here.
Let $(\lambda_n)_{n=1}^\infty$ be a strictly increasing, divergent sequence of positive numbers that preserves the addition.
By Remark \ref{sameexp}, we can use, from the beginning, expansion \eqref{f-log2} for the force.
We will also use the replacements indicated in Theorem \ref{logcase}(ii).

(a) When $m=5$, $Q_0(z_1,z_2,\cdots,z_5)=3z_1^2 z_3-2 z_2 z_5^4$, $Q_0(t)=t$, and $\beta=1$, if
 \beqs
 f(t)\dsim{\Phi} \sum_{n=1}^\infty \phi_n \big[3(\ln t)^2 L_3(t)  - 2 L_2(t)L_5(t)^4\big]^{-\lambda_n} \quad \text{in }G_{\alpha,\sigma},
 \eeqs
 where  $\Phi=\big((\ln t)^{-2\lambda}L_3(t)^{-\lambda}\big)_{\lambda>0}$,
 then 
 \beqs
 u(t)\dsim{\Phi} \sum_{n=1}^\infty \xi_n \big[3(\ln t)^2 L_3(t)  - 2 L_2(t)L_5(t)^4\big]^{-\lambda_n}\quad \text{in }G_{\alpha+1-\rho,\sigma}\text{ for all } \rho \in (0,1).
 \eeqs

(b) When $m=7$, $Q_0(z_1,z_2,\cdots,z_7)=4z_2-z_7^5+3$, $Q_1(t)=t^3-3t+1$, and $\beta=1/2$, if
 \beqs
 f(t)\dsim{\Phi} \sum_{n=1}^\infty \phi_n \big[4 L_2(t^{3/2}-3 t^{1/2}+1)  -  L_7(t^{3/2}-3t^{1/2}+1)^5 +3\big]^{-\lambda_n} \quad \text{in }G_{\alpha,\sigma},
 \eeqs
 where  $\Phi=(L_2(t)^{-\lambda})_{\lambda>0}$,
then 
 \beqs
 u(t)\dsim{\Phi}  \sum_{n=1}^\infty \xi_n \big[4 L_2(t^{3/2}-3 t^{1/2}+1)  -  L_7(t^{3/2}-3t^{1/2}+1)^5 +3\big]^{-\lambda_n}\quad \text{in }G_{\alpha+1-\rho,\sigma}\text{ for all } \rho \in (0,1).
 \eeqs
\end{example}
\begin{example} \label{sin-eg}
Given $m \in \N$.
Consider the system
\beqs
\Psi=(\psi_\lambda)_{\lambda>0}=\left(\left[\sin (L_m^{-1}(t))\right]^{\lambda} \right)_{\lambda>0}.
\eeqs

Clearly, $\sin (L_m^{-1}(t))$ is  
a positive function defined on $[T_*,\infty)$ for some $T_* > 0$ and $\sin (L_m^{-1}(t))\to 0$ as $t\to\infty$. Thus $\Psi$ is a  P-system.

Since $\sin^{\lambda}(x)$ is increasing in $x$ on $(0,\pi/2)$, it implies that $\left[\sin (L_m^{-1}(t))\right]^{\lambda}$ is decreasing on $[T_1,\infty)$ for $T_1$ large enough.
Note that    
\beqs
\frac{2x}{\pi} \le \sin(x) \le x, \quad  0  \le  x \le \pi/2,
\eeqs
 one has $\sin (L_m^{-1}(t)) \Oeq L_m^{-1}(t)$, since $0 < L_m^{-1}(t) \le 1$ for large $t$.

This and Lemma \ref{lemL} yield
\beqs
\left[\sin (L_m^{-1}(at))\right]^{\lambda} \Oeq \left[L_m^{-1}(at)\right]^{\lambda}  \Oeq \left[L_m^{-1}(t)\right]^{\lambda} \Oeq \left[\sin (L_m^{-1}(t))\right]^{\lambda}. 
\eeqs
Clearly,
\beqs
\lim_{ t \to \infty} \frac{e^{-\alpha t }}{\psi_{\lambda}(t)}= \lim_{ t \to \infty} e^{-\alpha t }L_m(t)^{\lambda}= 0. 
\eeqs
Therefore, $\Psi$ satisfies Condition \ref{weaksys}.
We write
\begin{align*}
\psi'_\lambda(t)= -\lambda\left(\sin ( L_m^{-1}(t)  \right)^{\lambda-1}\cos( L_m^{-1}(t) ) \frac{1}{t \left(\Pi_{p=1}^{m-1}L_p(t) \right) L_m(t)^2},
\end{align*}
and estimate
\begin{align*}
|\psi'_\lambda(t)|=\mathcal{O} \left( t^{-1} (\sin ( L_m^{-1}(t))^{\lambda-1}\right)
=\mathcal{O} \left( t^{-1} L_m(t)^{-\lambda+1}\right)
=\mathcal{O} \left(L_m(t)^{-\mu}  \right) = \mathcal{O} (\psi_{\mu}(t))
\end{align*}
for all $\mu > 0$.
Then $\Psi$ satisfies (ii) of Condition \ref{sysdef} with $N_\lambda=0$ for all $\lambda>0$, and thus, all parts of  Condition \ref{sysdef} due to Property (P).
Let $(\lambda_n)_{n=1}^\infty$ be as in Example \ref{log-eg}. By Theorem \ref{mainthm},
if
\beqs \label{f-sin}
f(t)\simP \sum_{n=1}^{\infty} \phi_n \left[\sin (L_m^{-1}(t))\right]^{\lambda_n} \quad \text{in }G_{\alpha,\sigma},
\eeqs
then
\beqs \label{u-sin}
u(t)\simP \sum_{n=1}^{\infty} \xi_n \left[\sin (L_m^{-1}(t))\right]^{\lambda_n} \quad \text{in }G_{\alpha+1-\rho,\sigma}\text{ for all } \rho \in (0,1),
\eeqs
with $\xi_n$'s being \eqref{xin-log}.
\end{example}
\begin{example}\label{tan-eg}
Given $m \in \N$.
Consider the system
\beqs
\Psi=(\psi_\lambda)_{\lambda>0}=\left(\left[\tan (L_m^{-1}(t))\right]^{\lambda} \right)_{\lambda>0}.
\eeqs
Similar to Example \ref{sin-eg}, using the fact that
\beqs
x \le \tan(x) \le 2x, \quad  0  \le  x \le 1/2,
\eeqs
one can verify that $\Psi$ satisfies Condition \ref{weaksys} and Condition \ref{sysdef} with $N_\lambda=0$ for all $\lambda>0$. 

Again, let $(\lambda_n)_{n=1}^\infty$ be as in Example \ref{log-eg}. 
We obtain that if
\beqs \label{f-tan}
f(t)\simP \sum_{n=1}^{\infty} \phi_n \left[\tan (L_m^{-1}(t))\right]^{\lambda_n} \quad \text{in }G_{\alpha,\sigma},
\eeqs
then
\beqs \label{u-tan}
u(t)\simP \sum_{n=1}^{\infty} \xi_n \left[\tan (L_m^{-1}(t))\right]^{\lambda_n} \quad \text{in }G_{\alpha+1-\rho,\sigma}\text{ for all } \rho \in (0,1),
\eeqs
with $\xi_n$'s from \eqref{xin-log}.

\end{example}

\subsection{A system with infinite expansions for the derivatives}\label{ss3}

In the previous two subsections, the expansion \eqref{ppex} is zero or a finite sum.
In this subsection, we demonstrate the case when each expansion \eqref{ppex} is an infinite sum. 

Consider a particular P-system $\Psi=(\psi_\lambda)_{\lambda>0}$ with $\psi_\lambda=(\sqrt{t}+1)^{-\lambda}$. 
Let $\lambda>0$. We see, for any $t>0$,  that
\begin{align*}
\psi_\lambda'(t)
&=-\lambda (\sqrt t+1)^{-\lambda-1 }\frac12\frac{1} {\sqrt t}
=-\frac{\lambda}{2} (\sqrt t+1)^{-\lambda-1 }\frac{1} {\sqrt t+1}\cdot \frac{1}{1-\frac{1}{\sqrt t+1}}\\
&=-\frac{\lambda}{2} (\sqrt t+1)^{-\lambda-1 }\sum_{k=1}^\infty \frac{1} {(\sqrt t+1)^k}
=\sum_{k=1}^\infty -\frac{\lambda}{2} (\sqrt t+1)^{-\lambda-k-1 }.
\end{align*}

Applying Lemma \ref{power-asym} to $X=\R$,  $\varphi(t)=(\sqrt t+1)^{-1}$, $\lambda_n=\lambda+n+1$, $M=2$, $\xi_n=-\lambda/2$, $c_0=\lambda/2$, $\kappa=1$, 
we deduce that the derivative $\psi_\lambda'(t)$, in fact, has the expansion 
\beqs
\psi_\lambda'(t) \simP \sum_{k=1}^\infty -\frac{\lambda}{2} (\sqrt t+1)^{-\lambda-k-1 }.
\eeqs

Thus, we have expansion \eqref{ppex} with 
\beqs
N_\lambda=\infty,\text{ and}\quad c_{\lambda,k}=-\lambda/2,\quad \lambda^\vee(k)=\lambda+1+k\text{ for all }k\in\N.
\eeqs

With this, Property (P), and some elementary estimates, we can verify that Conditions \ref{sysdef} and \ref{weaksys} are met.

We assume the force has the expansion
\beqs
f(t)\simP \, \sum_{n=1}^\infty \tilde \phi_n \psi_{\gamma_n}(t) =\sum_{n=1}^\infty \tilde \phi_n (\sqrt{t}+1)^{-\gamma_n} \text{ in }G_{\alpha,\sigma}.
\eeqs

Define
\beqs
S_*=\Big \{  (\sum_{j=1}^p \gamma_{n_j}) + k: p,n_1,n_2,\ldots,n_p\in\N, \
k=0 \text{ or }k\in\N\cap[2,\infty)\Big  \}. 
\eeqs

Then the set $S_*$ satisfies Assumption \ref{C1}, and we can assume \eqref{museq} and the expansion 
\beq \label{f-root2}
f(t) \simP \sum_{n=1}^{\infty} \phi_n (\sqrt{t}+1)^{-\lambda_n }\quad \text{in }G_{\alpha,\sigma},
\eeq
for a sequence $(\phi_n)_{n=1}^\infty$  in $G_{\alpha,\sigma}$.
We apply Theorem \ref{mainthm} and obtain the following.

\begin{proposition}\label{rootcase}
Assume \eqref{f-root2}. Then any Leray-Hopf weak solution $u(t)$ of the NSE \eqref{fctnse} admits the asymptotic expansion
\beqs 
u(t) \simP \sum_{n=1}^\infty  \xi_n (\sqrt{t}+1)^{-\lambda_n}\quad \text{in }G_{\alpha+1-\rho,\sigma}\text{ for all } \rho \in (0,1),
 \eeqs
where 
\begin{align*} 
\xi_1=A^{-1}\phi_1,\quad 
\xi_n&=A^{-1}\Big(\phi_n + \frac12\sum_{ p \in \mathcal{Z}_n} \lambda_p \xi_p - \sum_{\stackrel{1\le k,m\le n-1,}{\lambda_k+\lambda_m=\lambda_n}}B(\xi_k,\xi_m)\Big) \quad\text{for } n\ge 2,
\end{align*}
with
$
\mathcal{Z}_n= \{ p \in \N\cap [1, n-1]: \exists k\in\N, \lambda_p+1+k=\lambda_n\}.
$
\end{proposition}

\subsection{Expansions using a background system} \label{ss4}

In this subsection, we present a scenario for which the use of the background systems in Section \ref{discrtsec} is essential.
To motivate our more general force $f$ later, we consider a simple case first.
Let $\gamma\in(0,1)$, $\beta_0>0$, and
\beqs
f(t)= \frac{\phi}{(t^\gamma+1)^{\beta_0}}\quad\text{with } \phi\in G_{\alpha,\sigma}.
\eeqs

Then we expect a solution $u(t)$ of the NSE \eqref{fctnse} to have an asymptotic expansion containing at least $(t^\gamma+1)^{-\beta}$ for some $\beta>0$.
The derivative term $u_t$ in the NSE will contain $\ddt (t^\gamma+1)^{-\beta}$, which is 
\begin{align*}
\ddt (t^\gamma+1)^{-\beta}&=\frac{ -\gamma \beta}{(t^\gamma+1)^{ \beta +1}t^{1-\gamma} }
=\frac{ -\gamma \beta}{(t^\gamma+1)^{ \beta +1}(t^{1-\gamma}+1)(1-\frac1{t^{1-\gamma}+1}) }\\
&=\frac{ -\gamma \beta}{(t^\gamma+1)^{ \beta +1}  } \sum_{k=1}^\infty \frac{1} {(t^{1-\gamma}+1)^k},
\end{align*}
thus,
\beq\label{dtgb}
\ddt (t^\gamma+1)^{-\beta}
= \sum_{k=1}^\infty \frac{ -\gamma \beta} {(t^\gamma+1)^{ \beta +1} (t^{1-\gamma}+1)^k}.
\eeq

Thanks to the term $Au$ in the NSE, \eqref{dtgb} in turn suggests that a possible asymptotic expansion of $u(t)$ may have to include infinitely many terms $(t^\gamma+1)^{-\lambda}(t^{1-\gamma}+1)^{-\mu}$.

Because of this, we now consider a function
$\psi(t)=(t^\gamma+1)^{-\lambda}(t^{1-\gamma}+1)^{-\mu}$.
Taking the derivative by the product rule gives
\beqs
\psi'(t)=\Big(\ddt \frac{1}{(t^\gamma+1)^\lambda}\Big)\frac{1}{(t^{1-\gamma}+1)^\mu}
+\frac{1}{(t^\gamma+1)^\lambda}\Big(\ddt \frac{1}{(t^{1-\gamma}+1)^\mu}\Big).
\eeqs

Using \eqref{dtgb} with $\beta:=\lambda$ for the first derivative, and with $\gamma:=1-\gamma$, $\beta:=\mu$ for second derivative, 
we obtain
\beq\label{psip}
 \psi'(t)
= \sum_{k=1}^\infty \frac{ -\gamma \lambda} {(t^\gamma+1)^{ \lambda +1} (t^{1-\gamma}+1)^{\mu+k}}
 + \sum_{k=1}^\infty \frac{ -(1-\gamma) \mu} {(t^{1-\gamma}+1)^{ \mu +1} (t^\gamma+1)^{\lambda+k}}.
\eeq

Observe that the sums in \eqref{psip} involve the functions of \underline{the same form} as $\psi$, but with \underline{different} powers.
Also, the equality can be converted, under proper conditions, to an asymptotic expansion with the  background system $\Phi=(t^{-\lambda})_{\lambda>0}$.

\medskip
\noindent\textbf{Fixing a background system.} Let us fix the P-system $\Phi=(\varphi_\lambda)_{\lambda>0}$, where $\varphi_\lambda(t)=t^{-\lambda}$.
By (i)--(iii) in subsection \ref{ss1}, we see that $\Phi$ satisfies 
Condition (iii) of Assumption \ref{dsysdef}.

\medskip
From the above observation, we consider a force $f$ having the following general expansion
\beq\label{f-prod}
f(t) \dsim{\Phi} \sum_{n=1}^{\infty} \tilde\phi_n \tilde\psi_n(t)=\sum_{n=1}^{\infty}  \frac{\tilde \phi_n}{(t^\gamma+1)^{\tilde \alpha_n}(t^{1-\gamma}+1)^{\tilde \beta_n} } \quad \text{ in }G_{\alpha,\sigma},
\eeq
where $\tilde\psi_n(t)= (t^\gamma+1)^{-\tilde \alpha_n}(t^{1-\gamma}+1)^{-\tilde \beta_n}$,
$\gamma$ is a constant in the interval $(0,1)$, $(\tilde \alpha_n)_{n=1}^\infty$ and $(\tilde \beta_n)_{n=1}^\infty$ are sequences of non-negative numbers
such that 
\beq\label{lamtil} \tilde\lambda_n\eqdef \gamma\tilde\alpha_n+(1-\gamma)\tilde\beta_n\text{ is positive, strictly increasing (in $n$) to infinity}. 
\eeq 

Note that the expansion in \eqref{f-prod} is understood in the sense of Definition \ref{dpremsys} with $\tilde\Psi=(\tilde\psi_n)_{n=1}^\infty$ replacing $\Psi$, and $\tilde\lambda_n$ replacing $\lambda_n$. 

A simple example of \eqref{f-prod} is a finite sum $f(t)= \sum_{n=1}^N \tilde\phi_n \tilde\psi_n(t)$ for some $N\in\N$.
For more complicated cases of infinite sums, see Corollary \ref{cor79} below.

\begin{assumption}\label{gamassm}
The number  $\gamma$ is irrational, while  numbers $\tilde \alpha_n$ and $\tilde \beta_n$ are rational for all $n\in\N$. 
\end{assumption}

Define the sets
\begin{align}
\label{Ss1} 
S_*&=\Big \{  \gamma(\sum_{j=1}^p \tilde \alpha_{n_j} + k) + (1-\gamma)(\sum_{j=1}^p \tilde\beta_{n_j} + \ell) : \notag \\
&\qquad p,n_1,n_2,\ldots,n_p\in\N,\ (k,\ell)\in \N^2\cup (0,0)\Big  \},\\
E_1&=\Big\{\sum_{j=1}^p \tilde \alpha_{n_j} + k: p,n_1,n_2,\ldots,n_p\in\N,\ k\in\N\cup \{0\}\Big\}, \notag \\
E_2 &=\Big\{ \sum_{j=1}^q \tilde\beta_{k_j} + \ell:q,k_1,k_2,\ldots,k_q\in\N,\ \ell\in\N\cup \{0\}\Big\}.\notag
\end{align}

Note that 
\beq\label{Ssubar}
 \{\tilde \lambda_n:n\in\N\}\subset S_*\subset E^*\eqdef \{ \gamma\alpha+(1-\gamma)\beta: \alpha\in E_1,\ \beta\in E_2 \}.
\eeq

We see that $S_*$ and $E^*$ preserve the addition. 


\begin{lemma}\label{unique}
For each $\mu\in E^*$, there exists a unique pair $(\alpha,\beta)\in E_1\times E_2 $ such that
\beq\label{mudecomp}
\mu=\gamma\alpha+(1-\gamma)\beta.
\eeq 
\end{lemma}
\begin{proof}
The existence of the decomposition \eqref{mudecomp} comes directly from  \eqref{Ssubar}.
We prove the uniqueness now. Let $\mu\in E^*$ and
suppose there are $(\alpha,\beta),(\alpha',\beta')\in E_1\times E_2 $ such that
$$\mu=\gamma\alpha+(1-\gamma)\beta=\gamma\alpha'+(1-\gamma)\beta'.$$
Then
$\gamma(\alpha-\alpha')=-(1-\gamma)(\beta-\beta').$
Note from this relation  that $\alpha=\alpha'$ if and only if $\beta=\beta'$.

Consider the case $\alpha\ne\alpha'$ and $\beta\ne \beta'$. Then
\beq\label{gamfrac}
\frac\gamma{1-\gamma}=-\frac{\beta-\beta'}{\alpha-\alpha'}.
\eeq

We have $\alpha-\alpha'=\sum_{j=1}^p \pm \tilde \alpha_{n_j} + k\ne 0$, and $\beta-\beta'=\sum_{j=1}^q \pm \tilde\beta_{k_j} + \ell\ne 0$, for some $k,\ell\in\Z$, and some particular $+/-$ signs in the two sums.

Since $\gamma$ is irrational, so is $\gamma/(1-\gamma)$, while the right-hand side of \eqref{gamfrac}  is rational, which yields  a contradiction. Therefore, we can only have $\alpha=\alpha'$ and $\beta=\beta'$. 
\end{proof}

We rewrite $S_*$ as
\beq\label{Snew}
S_*=\Big \{  \sum_{j=1}^p \tilde \lambda_{n_j} + \gamma k+ (1-\gamma)\ell :
 p,n_1,n_2,\ldots,n_p\in\N,\ (k,\ell)\in \N^2\cup (0,0)\Big  \}.
\eeq

Since $\tilde\lambda_n\to\infty$, it follows \eqref{Snew} that we can order $S_*$ to be a sequence $(\lambda_n)_{n=1}^\infty$ as in \eqref{museq}. 

Note that $\lambda_n\to\infty$, hence consequently, $\alpha_n+\beta_n\to\infty$.

\medskip
\noindent\textbf{The discrete system for expansions.} Let $\Psi=(\psi_n)_{n=1}^\infty$, where 
\beqs
\psi_n(t)=(t^\gamma+1)^{-\alpha_n}(t^{1-\gamma}+1)^{-\beta_n},
\eeqs
with  $(\alpha_n,\beta_n)$, thanks to Lemma \ref{unique},  being the unique pair in $E_1\times E_2$ such that
\beq\label{mu2}
\lambda_n=\gamma\alpha_n+(1-\gamma)\beta_n.
\eeq 

Clearly, 
$\psi_n(t)\Oeq t^{-\lambda_n}=\varphi_{\lambda_n}(t)$.
Hence, $\Psi$ and $\Phi$ satisfy condition \eqref{psph} in Definition \ref{dpremsys}.

\medskip
We still need to verify the remaining Conditions (i) and (ii) of Assumption \ref{dsysdef}.

\medskip
\noindent\textit{Verification of Condition {\rm(i)}.}
For $m,n \in \N$, we have
\beq\label{L1}
\psi_m(t)=(t^\gamma+1)^{-\alpha_m}(t^{1-\gamma}+1)^{-\beta_m},\quad \lambda_m=\gamma\alpha_m+(1-\gamma)\beta_m\in S_*,
\eeq
\beq\label{L2}
\psi_n(t)=(t^\gamma+1)^{-\alpha_n}(t^{1-\gamma}+1)^{-\beta_n},\quad \lambda_n=\gamma\alpha_n+(1-\gamma)\beta_n\in S_*,
\eeq
where $\alpha_m,\alpha_n\in E_1$ and $\beta_m,\beta_n\in E_2$.
Then
\beqs
(\psi_m\psi_n)(t)=(t^\gamma+1)^{-\alpha_m-\alpha_m}(t^{1-\gamma}+1)^{-\beta_n-\beta_n}.
\eeqs

Since $S_*$ preserves the addition we have $\lambda_m+\lambda_n\in S_*$, hence there exists $k$ such that
\beq \label{kmn}
\lambda_k=\lambda_m+\lambda_n.
\eeq 

By \eqref{kmn}, \eqref{L1} and \eqref{L2}, we have $\lambda_k=\gamma(\alpha_m+\alpha_n)+(1-\gamma)(\beta_m+\beta_n)$.
Because $\alpha_m+\alpha_n\in E_1$, $\beta_m+\beta_n\in E_2$, and by the uniqueness of the decomposition of $\lambda_k$,
we deduce
$\alpha_m+\alpha_n=\alpha_k$ and $\beta_n+\beta_m=\beta_k$.
Therefore,
\beqs
(\psi_m\psi_n)(t)=(t^\gamma+1)^{-\alpha_k}(t^{1-\gamma}+1)^{-\beta_k}=\psi_k(t).
\eeqs

This proves that \eqref{dtriple} of Assumption \ref{dsysdef} holds true with $d_{m,n}=1$ and $k=m\wedge n\in\N$ satisfying \eqref{kmn}.

Thus, Condition (i) of Assumption \ref{dsysdef} is met.

\medskip
\noindent\textit{Verification of Condition {\rm(ii)}.}
Using \eqref{psip}, we have, for $n\in\N$, $t>0$, 
\begin{align*}
 \psi_n'(t)
&= \sum_{k=1}^\infty \frac{ -\gamma \alpha_n} {(t^\gamma+1)^{ \alpha_n +1} (t^{1-\gamma}+1)^{\beta_n+k}}
 + \sum_{k=1}^\infty \frac{ -(1-\gamma) \beta_n} {(t^\gamma+1)^{\alpha_n+k}(t^{(1-\gamma}+1)^{ \beta_n +1} }\\
&=G_{1,n}(t)+G_{2,n}(t), 
\end{align*}
where
\beqs
G_{n,1}(t)=\sum_{k=1}^\infty  \frac{ -\gamma \alpha_n} {(t^\gamma+1)^{ \alpha_n +1} (t^{1-\gamma}+1)^{\beta_n+k}},
\eeqs
\beqs
G_{n,2}(t)=
\sum_{k=1}^\infty \frac{ -(1-\gamma) \beta_n} {(t^\gamma+1)^{\alpha_n+k}(t^{(1-\gamma}+1)^{ \beta_n +1} }.
\eeqs

Let $n\in\N$ be fixed momentarily. We apply Lemma \ref{lem710} to 
$$\varphi(t)=t^{-1},\ 
\bar \psi_k(t)=(t^\gamma+1)^{-\alpha_n-1} (t^{1-\gamma}+1)^{-\beta_n-k},\
T_*=1,$$
$$X=\R,\ \bar \xi_k=-\gamma\alpha_n,\  \bar\nu_k=\gamma(\alpha_n+1)+(1-\gamma)(\beta_n+k).$$

Note, for $t\ge 1$, that
$$  2^{-\alpha_n-\beta_n-k-1}t^{-\bar\nu_k}=(2t^\gamma)^{-\alpha_n-1} (2t^{1-\gamma})^{-\beta_n-k} 
\le \bar \psi_k(t)\le t^{-\bar\nu_k},$$
which yields
$$ D_k^{-1}\varphi(t)^{\bar\nu_k} \le \bar \psi_k(t)\le D_k \varphi(t)^{\bar\nu_k},$$
where $D_k= 2^{\alpha_n+\beta_n+k+1}$.
Taking  $c_0=\gamma\alpha_n$, $\kappa=1$, and $M=3^{1/(1-\gamma)}$, we have
\beqs
|\xi_k|\le c_0 \kappa^{\bar\nu_k},\quad \sum_{k=1}^\infty D_k M^{-\bar\nu_k}<\infty.
\eeqs

Since $(\bar\nu_k)_{k=1}^\infty$ is already strictly increasing, it is its own strictly increasing re-arrangement.

Then, by Lemma \ref{lem710},
\beq\label{G1}
G_{n,1}(t)\dsim{\Phi} \sum_{k=1}^\infty \bar\xi_k\bar\psi_k(t)= \sum_{k=1}^\infty \frac{ -\gamma \alpha_n} {(t^\gamma+1)^{ \alpha_n +1} (t^{1-\gamma}+1)^{\beta_n+k}}.
\eeq

Now, with 
$\hat\xi_k=-(1-\gamma) \beta_n$ replacing $\xi_k$,
$\hat\psi_k(t)=(t^\gamma+1)^{-\alpha_n-k}(t^{1-\gamma}+1)^{-\beta_n -1}$ replacing $\bar\psi_k(t)$,
and $\hat\nu_k=\gamma(\alpha_n+k)+(1-\gamma)(\beta_n+1)$ replacing $\bar\nu_k$,
we similarly obtain
\beq\label{G2}
G_{n,2}(t)\dsim{\Phi}\sum_{k=1}^\infty \hat\xi_k\hat\psi_k(t)
=\sum_{k=1}^\infty \frac{ -(1-\gamma) \beta_n} {(t^\gamma+1)^{\alpha_n+k}(t^{1-\gamma}+1)^{ \beta_n +1} }.
\eeq

Since $\bar\nu_k$ and $\hat\nu_k$ belong to $S_*$, all functions $\bar\psi_k(t)$ in \eqref{G1}, and $\hat\psi_k(t)$ in \eqref{G2} belong to the collection $\{\psi_n:n\in\N\}$.
Then we can rewrite 
\beqs
G_{n,i}(t)\dsim{\Phi} \sum_{k=1}^\infty \tilde c_{n,i,k}\psi_k(t),\quad\text{for } i=1,2,
\eeqs
for some constants $\tilde c_{n,i,k}$.
Therefore,
\beq\label{pc-nk}
\psi_n'(t)=G_{n,1}(t)+G_{n,2}(t)\dsim{\Phi} 
-\sum_{k=1}^\infty c_{n,k}\psi_k(t),
\eeq
where
\beq\label{newcnk}
c_{n,k}=-(\tilde c_{n,1,k}+\tilde c_{n,2,k})
=\gamma \sum_j\alpha_j +(1-\gamma)\sum_\ell\beta_\ell,
\eeq
with $\alpha_j+1=\alpha_k$, $\beta_j+p=\beta_k$, and 
$\alpha_\ell+q=\alpha_k$, $\beta_\ell+1=\beta_k$ for some $p,q\in \N\cup\{0\}$.

Note that these pairs $(j,p)$ and $(\ell,q)$ are only finitely many. 
Indeed, since 
$$\alpha_j+\beta_j+p+1=\alpha_k+\beta_k\text{ and }\alpha_j+\beta_j\to\infty,
$$
we have, for each fixed $k$, there are only finitely many $j$ and $p$.
The same arguments apply to  $(\ell,q)$. Therefore, the sums in \eqref{newcnk} are only finite ones.

With \eqref{pc-nk}, Condition (ii) of Assumption \ref{dsysdef} is met.

\medskip
\noindent\textit{Conclusion on Assumption \ref{dsysdef}.} We have checked that the systems $\Psi$ and $\Phi$ satisfy Assumption \ref{dsysdef}.

\medskip
We now return to the expansion \eqref{f-prod} for the force. Since $\tilde\lambda_n\in S_*$, we have that each function $(t^\gamma+1)^{-\tilde \alpha_n}(t^{1-\gamma}+1)^{-\tilde \beta_n}$ in the sum in \eqref{f-prod}
belongs to $\{\psi_n:n\in\N\}$.
Hence, we can rewrite \eqref{f-prod} as 
\beq\label{f-prod2}
f(t)\dsim{\Phi} \sum_{n=1}^\infty \phi_n \psi_n(t)
=\sum_{n=1}^{\infty} \frac{ \phi_n}{(t^\gamma+1)^{ \alpha_n}(t^{1-\gamma}+1)^{ \beta_n} } \quad \text{in }G_{\alpha,\sigma},
\eeq
for some sequence $(\phi_n)_{n=1}^\infty$ in $G_{\alpha,\sigma}$.

By applying Theorem \ref{discretethm}, we obtain the following result.

\begin{proposition}\label{productcase}
Assume \eqref{f-prod2}. Then any Leray-Hopf weak solution $u(t)$ of the NSE \eqref{fctnse} admits the asymptotic expansion
\beqs
u(t) \dsim{\Phi} \sum_{n=1}^{\infty} \xi_n\psi_n(t)=\sum_{n=1}^{\infty} \frac{ \xi_n}{(t^\gamma+1)^{ \alpha_n}(t^{1-\gamma}+1)^{ \beta_n} } \quad \text{in }G_{
\alpha +1-\rho,\sigma}\text{ for all } \rho \in (0,1),
 \eeqs
where 
\beqs 
\xi_1=A^{-1}\phi_1,\quad
\xi_n=A^{-1}\Big(\phi_n + \sum_{ p=1}^{n-1} c_{p,n} \xi_p - \sum_{\stackrel{1\le k,m\le n-1,}{\lambda_k+\lambda_m=\lambda_n}}B(\xi_k,\xi_m)\Big) \quad\text{for } n\ge 2,
\eeqs
with $c_{p,n}$ being defined in \eqref{newcnk}.
\end{proposition}

\begin{remark}\label{same2}
This is a counterpart of Remark \ref{sameexp}, but applied to the expansion \eqref{f-prod}.

Assume $(\tilde\alpha_n)_{n=1}^\infty$ and $(\tilde\beta_n)_{n=1}^\infty$ are sequences of non-negative numbers such that
\begin{enumerate}[label={\rm(\alph*)}]
 \item  Property \eqref{lamtil} holds true,
 \item For each $n\in\N$, the right-hand side of \eqref{lamtil} is the unique decomposition among $\gamma \tilde\alpha_k+(1-\gamma) \tilde\beta_j$ for $k,j\in\N$,
 \item Each set
$E_{*,1}=\{\tilde\alpha_n:n\in\N\}$, $E_{*,2}=\{\tilde\beta_n:n\in\N\}$
preserves the addition, and the increments by $1$,
 \item The set 
$E_*\eqdef \{\tilde\lambda_n:n\in\N\}$ is equal to $\bar E\eqdef \{\gamma \tilde\alpha_k+(1-\gamma) \tilde\beta_j:k,j\in\N\}$.
\end{enumerate}


Define
 $E = \Big\{\gamma(\tilde\alpha_n+k)+(1-\gamma)(\tilde\beta_n+j):n\in\N,\ k,j\in\N\cup\{0\}\Big\}$. 

 By the preservation of $E_{*,1}$ and $E_{*,2}$ in (c), we have $\tilde\alpha_n+k\in E_{*,1}$ and $\tilde\beta_n+k\in E_{*,2}$.
Then   $E_*\subset E\subset \bar E=E_*$, which implies $E_*=E$. 
 
Let $S_*$ be defined by \eqref{Ss1}. We have $E_*\subset S_*\subset E=E_*$.
Hence, $S_*=E_*$, which, by \eqref{lamtil}, is already ordered by $(\tilde\lambda_n)_{n=1}^\infty$.
By this and \eqref{museq}, \eqref{mu2}, we have $\lambda_n=\tilde\lambda_n$, $\alpha_n=\tilde \alpha_n$, $\beta_n=\tilde \beta_n$, and 
$\psi_n=\tilde\psi_n$. Therefore, the expansion \eqref{f-prod2} is the original \eqref{f-prod}.
\end{remark}

\appendix
\section{Appendix}\label{apA}
One way to generate an infinite expansion of the type  \eqref{g0} is to start with a function as a convergent series in \eqref{gequal}.
We give a criterion for such a conversion.

\begin{lemma}\label{power-asym}
Let $\varphi$ be a positive function defined on $[T_*,\infty)$ for some $T_*\ge 0$, and $\varphi(t)\to\infty$ as $t\to\infty$.

Let $(\lambda_n)_{n=1}^\infty$ be a strictly increasing, divergent sequence of positive numbers and there exists a number $M>1$ such that 
\beq\label{ML}
\sum_{n=1}^\infty M^{-\lambda_n}<\infty.
\eeq

Let $(X, \|\cdot\|)$ be a  Banach space 
 and let $(\xi_n)_{n=1}^\infty \subset X$ satisfy
\beq \label{xin-bdd}
\|\xi_n\|  \le c_0\kappa^{\lambda_n} \quad\forall n\in\N,
\eeq
for some positive constants $c_0$ and  $\kappa$.

Then the series $\sum_{n=1}^\infty \xi_n \varphi(t)^{\lambda_n}$ converges absolutely and uniformly to a function $g(t)$ on $[T_0,\infty)$ for some $T_0 \ge T_*$, and 
$g$ has the expansion 
\beq\label{gex}
g(t)\simP \sum_{n=1}^\infty  \xi_n \varphi(t)^{\lambda_n} \text{ in }X.
\eeq
\end{lemma}
\begin{proof}
Since  $\varphi(t) \to 0 $ as $t\to\infty$, there exists $T_0\ge T_*$ such that  for all $ t \ge T_0$,
\beq \label{vphi-bdd}
\varphi(t) \le \frac1{M\kappa}.
\eeq

Combining \eqref{xin-bdd} and \eqref{vphi-bdd} yields, for all $n \in \N$,
\beqs
\sup_{[T_0,\infty)}\|\xi_n \varphi(t)^{\lambda_n}\| \le c_0 M^{-\lambda_n}.
\eeqs

This and \eqref{ML} imply that $\sum_{n=1}^\infty \xi_n \varphi(t)^{\lambda_n}$ converges absolutely and uniformly on $[T_0,\infty)$, with $g(t)=\sum_{n=1}^\infty \xi_n \varphi(t)^{\lambda_n}$ being its limit function. 

It remains to prove the expansion \eqref{gex}.
We note,  by the convergence of the series $\sum_{n=1}^\infty \xi_n \varphi(T_0)^{\lambda_n}$, that 
\beqs
\sup_{n\in\N} \|\xi_n\| \varphi(T_0)^{\lambda_n} = c_1<\infty,
\eeqs
which implies 
\beq \label{xin-bdd2}
\|\xi_n\|  \le \frac{c_1}{\varphi(T_0)^{\lambda_n}}\quad\forall n\in\N.
\eeq

Again, since $\varphi(t) \to 0 $ as $t\to\infty$, there exists $T_1\ge T_0$ such that for all $ t \ge T_1$,
\beq \label{vphi-Tnot}
\frac{\varphi(t)}{\varphi(T_0)} \le \frac{1}{M}.
\eeq

Using \eqref{xin-bdd2} and \eqref{vphi-Tnot}, we estimate, for all $t \ge T_1$,
\begin{align*}
\Big\|g(t)-\sum_{n=1}^N \xi_n \varphi(t)^{\lambda_n}\Big\|
&=\Big\| \sum_{n=N+1}^\infty \xi_n \varphi(t)^{\lambda_n}\Big\| \\
&\le c_1 \sum_{n=N+1}^\infty \frac{ \varphi(t)^{\lambda_n}}{\varphi(T_0)^{\lambda_n}} 
=  c_1  \,\frac{ \varphi(t)^{\lambda_{N+1}}}{\varphi(T_0)^{\lambda_{N+1}}} \sum_{n=N+1}^\infty \left(\frac{ \varphi(t)}{\varphi(T_0)}\right)^{\lambda_n-\lambda_{N+1}}\\
&\le  c_1  \,\frac{ \varphi(t)^{\lambda_{N+1}}}{\varphi(T_0)^{\lambda_{N+1}}} \sum_{n=N+1}^\infty M^{-\lambda_n+\lambda_{N+1}}
\le C_N\varphi(t)^{\lambda_{N+1}},
\end{align*}
where $C_N=c_1  (M/\varphi(T_0))^{\lambda_{N+1}} \sum_{n=1}^\infty M^{-\lambda_n}<\infty$.
Therefore, we obtain the expansion \eqref{gex}, according to Definition \ref{premsys} with $\psi_\lambda=\varphi^\lambda$.
\end{proof}

We extend Lemma \ref{power-asym} to cover the expansions with a background system such as those in Section \ref{discrtsec}.

\begin{lemma}\label{lem710}
Let $\varphi(t)$ and $\psi_n(t)$, for $n\in\N$, be positive functions defined on $[T_*,\infty)$ for some $T_*\ge 0$ that tend to zero as $t\to \infty$. Assume, for each $n\in\N$,  
there exist a numbers $D_n\ge 1$ such that 
\beq\label{A8}
D_n^{-1}\varphi(t)^{\lambda_n}\le \psi_n(t) \le D_n\varphi(t)^{\lambda_n}\quad\forall t\ge T_*,
\eeq
where $(\lambda_n)_{n=1}^\infty$ is a sequence of positive numbers and $\lambda_n\to\infty$ as $n\to\infty$.
Assume further that there exists $M>0$ such that 
\beq\label{DnM}
\sum_{n=1}^\infty D_n M^{-\lambda_n}<\infty.
\eeq 

Let $(X,\|\cdot\|)$ be a Banach space, and  $(\xi_n)_{n=1}^\infty$ be a sequence in $X$ such that \eqref{xin-bdd} holds. 

{\rm(i)} Then the series in $\sum_{n=1}^\infty \xi_n \psi_n(t)$ converges absolutely and uniformly on $[T_0,\infty)$ for some $T_0\ge T_*$.
Define
\beq\label{f3sum}
f(t)=\sum_{n=1}^\infty \xi_n \psi_n(t)\in X \quad\forall t\ge T_0.
\eeq

{\rm(ii)} Assume the mapping $n\mapsto\lambda_n$ is one-to-one. 
Let $(\mu_n)_{n=1}^\infty$ be the strictly increasing re-arrangement of $(\lambda_n)_{n=1}^\infty$.
Define $\psi_n^*=\psi_k$  and $\xi_n^*=\xi_k$ with $\mu_n=\lambda_k$. Let  $\Phi=(\varphi(t)^{-\lambda})_{\lambda>0}$.  Then
\beq\label{f3ex}
f(t)\dsim{\Phi}\sum_{n=1}^\infty \xi_n^* \psi_n^*(t).
\eeq
\end{lemma}
\begin{proof}
(i) There is $T_0\ge T_*$ such that 
$\varphi(t)\le 1/\kappa M$ for all $t\ge T_0$.
Then for all $t\ge T_0$, by \eqref{xin-bdd}, \eqref{A8},  and \eqref{DnM},
 \begin{align*}
  \sum_{n=1}^\infty \|\xi_n\| \psi_n(t)
  &\le \sum_{n=1}^\infty \|\xi_n\| D_n\varphi(t)^{\lambda_n}
  \le \sum_{n=1}^\infty c_0 D_n M^{-\lambda_n}<\infty.
 \end{align*}
Therefore, we obtain the absolute and uniform convergence on $[T_0,\infty)$.

(ii) Let $D^*_n=D_k$ with $\mu_n=\lambda_k$. After the re-arrangement we still have 
\beq\label{DkM} 
\sum_{n=1}^\infty D_n^* M^{-\mu_n}=\sum_{k=1}^\infty D_k M^{-\lambda_k}<\infty,
\eeq  
and, because of the absolute convergence, 
\beq\label{f3arr}
f(t)=\sum_{n=1}^\infty \xi_n^* \psi_n^*(t) \text{ for }t\ge T_0.
\eeq
 
For each $n\in\N$, let $k\in\N$ such that $\mu_n=\lambda_k$, then we have 
 $$(D_n^*)^{-1}\varphi^{\mu_n}=D_k^{-1}\varphi^{\lambda_k}\le \psi_n^*=\psi_k\le D_k\varphi^{\lambda_k}=D^*_n\varphi^{\mu_n}.$$ 
 
The convergence of $f(T_0)$ in \eqref{f3arr} implies that there exists $c_1>0$ such that 
 \beqs
  \|\xi_n^*\|\le c_1 \psi_n^*(T_0)^{-1}\le c_1 D_n^* \varphi(T_0)^{-\mu_n}\quad \forall n\in\N. 
 \eeqs

 Let $T_1\ge T_0$ such that $\varphi(t)/\varphi(T_0)\le 1/M^2$ for all $t\ge T_1$. 
Then, for $t\ge T_1$,
\begin{align*}
\Big\|f(t)-\sum_{n=1}^N \xi_n^* \psi_n^*(t)\Big\|
&=\Big\| \sum_{n=N+1}^\infty \xi_n^* \psi_n^*(t)\Big\|
\le  \sum_{n=N+1}^\infty  c_1 D_n^* \varphi(T_0)^{-\mu_n}\cdot D_n^*\varphi(t)^{\mu_n}\\
&\le c_1 (\varphi(t)/\varphi(T_0))^{\mu_{N+1}}\sum_{n=N+1}^\infty  (D_n^*)^2(\varphi(t)/\varphi(T_0))^{\mu_n-\mu_{N+1}}\\
&\le c_1 (\varphi(t)/\varphi(T_0))^{\mu_{N+1}}\sum_{n=N+1}^\infty  (D_n^*)^2M^{-2\mu_n+2\mu_{N+1}}\\
&\le c_1 (M^2\varphi(t)/\varphi(T_0))^{\mu_{N+1}}\Big(\sum_{n=N+1}^\infty  D_n^* M^{-\mu_n}\Big)^2
\le C\varphi(t)^{\mu_{N+1}},
\end{align*}
where, thanks to \eqref{DkM}, $C$ is a positive number.
Since $\mu_{N+1}>\mu_N$, we obtain \eqref{f3ex}.
\end{proof}

We emphasize that the sum in \eqref{f3sum} must be re-arranged to have a meaningful expansion as in \eqref{f3ex}.
In particular cases, Lemma \ref{lem710} is used to obtain expansions when $\psi_n$ is generated by two functions with two different sequences of powers.

\begin{corollary}\label{cor79}
Suppose $\zeta(t)$, $\vartheta(t)$ and $\varphi(t)$ are three positive functions defined on $[T_*,\infty)$ for some $T_*\ge 0$ that tend to zero as $t\to \infty$,  and there exist numbers $D\ge 1$, $s_1,s_2>0$ such that 
\beq\label{3rel}
D^{-1}\varphi(t)^{s_1}\le \zeta(t)\le D\varphi(t)^{s_1},\quad 
D^{-1}\varphi(t)^{s_2}\le \vartheta(t)\le D\varphi(t)^{s_2}
\quad\forall t\ge T_*.
\eeq

Let  $(\alpha_n)_{n=1}^\infty$ and $(\beta_n)_{n=1}^\infty$ be two sequences of non-negative numbers such that $\alpha_n+\beta_n\to\infty$. 

Define $\lambda_n=s_1 \alpha_n+s_2\beta_n$ for $n\in\N$.
Let $(X,\|\cdot\|)$ be a Banach space, and  $(\xi_n)_{n=1}^\infty$ be a sequence in $X$.
Assume  \eqref{ML} and \eqref{xin-bdd}. 

{\rm(i)}
Then the series in $\sum_{n=1}^\infty \xi_n \zeta(t)^{\alpha_n}\vartheta(t)^{\beta_n}$ converges absolutely and uniformly on $[T_0,\infty)$ for some $T_0\ge T_*$.
Define
\beqs
f(t)=\sum_{n=1}^\infty \xi_n \zeta(t)^{\alpha_n}\vartheta(t)^{\beta_n} \quad\forall t\ge T_0.
\eeqs

{\rm(ii)} Suppose the mapping $n\mapsto \lambda_n$ is one-to-one. 
Let $(\mu_n)_{n=1}^\infty$ be the strictly increasing re-arrangement of $(\lambda_n)_{n=1}^\infty$.
Define $\psi_n^*=\zeta^{\alpha_k}\vartheta^{\beta_k}$ and $\xi_n^*=\xi_k$ with 
$\mu_n=\lambda_k$.

Let $\Phi=(\varphi(t)^{-\lambda})_{\lambda>0}$.  Then
\beqs
f(t)\dsim{\Phi}\sum_{n=1}^\infty \xi_n^* \psi_n^*(t).
\eeqs
\end{corollary}
\begin{proof}
Let $\psi_n(t)=\zeta(t)^{\alpha_n}\vartheta(t)^{\beta_n}$ and $D_n=D^{\alpha_n+\beta_n}$.
 Thanks to \eqref{3rel}, we have $$D_n^{-1}\varphi^{\lambda_n}=D^{-(\alpha_n+\beta_n)}\varphi^{\lambda_n}\le \psi_n\le D^{\alpha_n+\beta_n}\varphi^{\lambda_n}=D_n\varphi^{\lambda_n}.$$
 
Note that $\lambda_n\to\infty$. 
Denote $s=\max\{1/s_1,1/s_2\}$. Then 
\beqs
 \sum_{n=1}^\infty D_n (M D^s)^{-\lambda_n}= \sum_{n=1}^\infty M^{-\lambda_n} D^{(1- s_1 s)\alpha_n+(1- s_2 s )\beta_n}
 \le \sum_{n=1}^\infty M^{-\lambda_n} <\infty.
 \eeqs
Hence, \eqref{DnM} holds true with $M:=MD^s$.
Applying Lemma \ref{lem710}, we obtain the desired statements in (i) and (ii).
\end{proof}

\def\cprime{$'$}\def\cprime{$'$} \def\cprime{$'$}

\end{document}